\documentclass[12pt]{amsart}
\usepackage{amsmath}
\usepackage{geometry,amsfonts,amssymb,amsthm,txfonts,pxfonts,amscd} 
\newcommand{\thm}[2]{\begin{#1} #2 \end{#1}}

\newcommand{\tr}{\mathrm{tr\:}}

\newcommand{\btheta}{\boldsymbol{\theta}}

\numberwithin{equation}{section}
\newtheorem{theorem}{Theorem}[section]
\newtheorem{lemma}[theorem]{Lemma}
\newtheorem{corollary}[theorem]{Corollary}
\newtheorem{conjecture}[theorem]{Conjecture}
\newtheorem{definition}[theorem]{Definition}

\theoremstyle{remark}
\newtheorem{observation}[theorem]{Observation}
\newtheorem{remark}[theorem]{Remark}

\newtheorem{question}[theorem]{Question}

\newcommand{\integers}{{\bf Z}}

\newcommand{\reals}{{\bf R}}

\newcommand{\calw}{\mathcal{W}}
\newcommand{\SL}{\mathrm{SL\:}}
\begin{document}

%-------------- Author entries --------------------

\title{Growth in free groups (and other stories)--twelve years later}
%Article title 
%\shorttitle{DRAFT} % Shortened version for
                                             % headline title 

\author{Igor Rivin}
\dedicatory{To Paul Schupp, with the greatest affection}
\address{Mathematics Department, Temple University, Philadelphia, PA 19122}
\address{School of Mathematics, Institute for Advanced Study, Princeton NJ 08540}
\email{rivin@temple.edu}
\email{rivin@ias.edu}
\date{\today}



\begin{abstract}
We start by studying the distribution of (cyclically
reduced) elements of the free groups 
$F_n$ with respect to their abelianization (or equivalently, their
class in $H_1(F_n, \integers)$). We derive an explicit generating function,
and a limiting distribution, by means of certain results (of
independent interest) on Chebyshev polynomials; we also prove that the
reductions $\mod p$ ($p$ -- an arbitrary prime) of these classes are
asymptotically equidistributed, and we study the deviation from
equidistribution. We extend our techniques to a more general setting
and use them to study the statistical properties of long cycles (and
paths) on regular (directed and undirected) graphs.  We return to the
free group to study  some growth functions of the number of conjugacy
classes as a function of their cyclically reduced length.
\end{abstract}
\maketitle

\section*{Introduction -- 2010}
The paper ``Growth in free groups (and other stories)", has been around in preprint form (\cite{RivinWalks1999}) since the late nineties (the arXiv version cited dates to 1999, but this was preceded by a 1997 IHES preprint). Since the paper has had a fair amount of influence (and  parts of it have since become separate papers), it seems a good idea to publish it at last -- this version is not very different from the preprint, except for this introduction, which gives a bit of background on how and why it was written together with a survey (necessarily incomplete and subjective) of  what has happened since the arXiv preprint appeared in 1999.

\subsection*{Why?}  The work described in the paper was initially motivated by the author's (continuing to this day) interest in the counting questions on geodesics on hyperbolic surface, stemming from some conversations with Peter Sarnak in the early 1990s. More precisely, Sarnak had asked about the asymptotics of the number of \emph{simple} geodesics on the punctured torus, where the only result appeared to be the one in the paper of Beardon, Lehner, and Sheingorn \cite{beardonlehnersheingorn}, where the authors had shown that the number of simple geodesics of length bounded by $L$ grew somewhere between quadratically and quartically in $L.$ This did not seem to be very sharp, and indeed, Greg McShane and I improved it to an asymptotic result (with quadratic growth) in a pair of short papers \cite{mc1,mc2}, using purely geometric methods (showing that the length of the \emph{unique} shortest geodesic  (which can be showed to be simple) in a primitive integral homology class extends to a norm on real homology (which is the Gromov, or the stable norm, though at the time McShane and I had no knowledge of the connection). The fact that there is at most one simple closed geodesic in a homology class is specific to the punctured torus, and while other methods can be used to compute the asymptotics of the number of simple closed geodesics of bounded length on a surface of finite type (the order of growth was computed by the author in \cite{geomded}, while asymptotics were computed by Maryam Mirzakhani in \cite{mirzakhcurves} -- see also \cite{rimizrakh}), the following question is still wide open:

\begin{citation}
{\bf Question:} How many simple curves of length bounded by $L$ are there in a fixed homology class $h$ on a hyperbolic surface?
\end{citation}
Mirzakhani's work implies that a constant proportion of all simple geodesics are separating, but for a non-trivial homology class nothing seems known to-date.

\subsection*{Geodesics in homology classes}
Given the interest in geodesics and homology, it was natural to investigate a similar question for \emph{all} closed geodesics, not necessarily simple. It is a well-known result of Huber (for hyperbolic surfaces -- Huber uses the Selberg Trace Formula)  -- \cite{hubergeod1,hubergeod2,hubergeod3} and Margulis \cite{margthes,margthesbook} for arbitrary negatively curved surfaces, using ergodic theory) that the number of closed geodesics of length bounded by $L$ \emph{without} homological restrictions is asymptotic to 
$\exp{hL}/(hL),$ where $h$ is the topological entropy of the geodesic flow ($h=1$ for a hyperbolic surface). The methods used by Huber and Margulis (Selberg Trace Formula and ergodic dynamics, respectively) are the two principal tools used in the vast majority of the paper discussed below (generally either one technique or the other, but not both, generally because the Trace Formula gets sharp results but only works in the constant curvature setting, while dynamical methods are softer, so give weaker results in a wider setting.

The first result on geodescis in homology classes is due to W. Parry and M. Pollicott -- in their paper \cite{ppchebotarev} they show that when the homology group $H_1(S, \integers)$ is \emph{finite}, then closed geodesics are equidistributed among homology classes. Parry and Pollicott use the machinery of thermodynamic formalism and dynamical zeta functions, and their argument mimics the proof of the Chebotarev density theorem. Parry and Pollicott's methods work in variable negative curvature, and they also analyze the lifting of geodesics in a homology class to (finite) Galois covers.  Roughly concurrently, A. Katsuda and T. Sunada showed in \cite{katsudasunada86} that for homology with coefficients in a finite group, every homology class contains an infinite number of closed geodesics (but no estimate of the growth of their number as a function of length).

The next result is due to T.~Adachi and T.~Sunada -- in the paper \cite{AdachiSunada1987} they show that the exponential growth rate of the number curves in any homology class is equal to $h$ (just like for homologically unrestricted geodesics) -- they use Markov partitions as introduced by R. Bowen in \cite{bowensymbolic} and use results on paths in finite graphs to get the result (which is rather weak, since they don't actually get an asymptotic result. They point out that getting such a result (via the usual Tauberian machinery) would require an understanding of the singularity of the $L$-functions involved greater than they could produce at the time.They conjecture that the the number of geodesics of length bounded by $L$ in a homology class should grow like $\exp(h L)/(L^{b+1}),$ where $b$ is the first Betti number of the manifold. 

This conjecture turns out to be false -- in the paper \cite{philsarnhomology}, published almost simultaneously with \cite{AdachiSunada1987}, R. Phillips and P. Sarnak give an asymptotic expansion valid for a \emph{hyperbolic} surface: the number of closed geodesics in a fixed homology class, of length bounded by $L$ grows as
\[
\dfrac{e^L}{L^{g+1}}(1+ c_1/L + c_2/L^2 + \cdots),
\]
where $c_1, \dotsc, c_k, \dotsc$ depend on the homology class. This sort of expansion appeared (at the time) to be possible only because the manifold had \emph{constant} negative curvature. The work of Phillips and Sarnak was extended (again, approximately at the same time) by C. L. Epstein to \emph{cusped} surfaces in \cite{charliegeod}, again using the Selberg Trace Formula. As often with these kinds of extensions, the result is a lot harder technically than the Phillips-Sarnak result.

At roughly the same time, A. Katsuda and T. Sunada extended the dynamical methods of \cite{AdachiSunada1987} first to surfaces of constant negtive curvature in \cite{katsudasunada88} (by observing that the complicated L-function that could not be dealt with in \cite{AdachiSunada1987} became much simpler in constant curvature), and then for general negatively curved surfaces in \cite{katsudasunada90}.

Last, but not least, S. Lalley uses the thermodynamical formalism and some fairly intricate harmonic analysis in \cite{lalleyhomology} to recover the results of Katsuda-Sunada, and more: He shows a central limit theorem for the distribution of homology classes of closed geodesics, and also a "large deviation result". Lalley's result is closest in spirit to the current paper, but the methods are completely different (and I had no knowledge of the paper's existence until this writing).

\subsection*{Some motivation}
All of the results mentioned in the survey above are technically quite involved, and it was not clear what was really going on. This is what gave birth to the current paper. One observation was that it is a lot easier to work with groups (especially free groups) than with surfaces, and secondly, since fundamental groups are often quasi-isometric to the spaces they are fundamental groups of, one has the hope of obtaining "universal" results (that is, a result for a surface group implies a  result (usually somewhat weaker) for every surface of the appropriate type.

One particular insight (on which much of the paper is based) is the observation that for graphs, the Selberg Trace Formula (quite pervasive in the work surveyed above) is a triviality: the number of closed (based) cycles of length $N$ in the graph is the trace of the $N$-th power of the adjacency matrix, and thus the sum of the $N$th powers of eigenvalues of the adjacency matrix of the graph. In the particular case where the graph is undirected, the adjacency matrix is symmetric, and analysis becomes easy.
Technically simpler methods (based in large part on perturbation theory for eigenvalues) have helped to get results of much wider scope than previously. Let us now review the results and their follow-up in subsequent years.

\subsection*{Then what happened? Free groups and related subjects}
In Section \ref{modsec} we have set up the basic model, and used it to count cyclically reduced words in a free group. The basic method works for any automatic group, and if the structure is bi-automatic, we similarly get an undirected graph. Somewhat surprisingly, the count of cyclically reduced words has been used in a number of papers (see, eg, \cite{KapovichSchupp,coornaertasymp}), and in the paper \cite{koganov} by L. M. Koganov it is shown that the formula is equivalent to H. Whitney's formula for the chromatic polynomial of the cycle graph. Koganov had apparently published two other papers (in 2002 and 2004) deriving the enumeration of cyclically reduced words -- see references [1] and [2] in \cite{koganov}.

A related question is considered in Sections \ref{gf}, \ref{gf2}, \ref{ihara}, where we study the number of \emph{conjugacy classes} of fixed minimal length in the free group (and elsewhere). We construct an ordinary generating function (in the form of a Lambert Series, see \cite{hardywright} for definition), which turns out to be horribly irrational (this result has gone on to have a life of its own in \cite{conj}), and the zeta function enumerating \emph{primitive} conjugacy classes, which turns out to be an Ihara-type zeta function of the defining graph (see also the papers of Stark and Terras \cite{staterI,staterII,staterIII}). The conjecture that the (standard) generating function is irrational for all non-virtually-cyclic Gromov-hyperbolic groups is still open. The Ihara zeta function immediately gives asymptotic growth rates for primitive classes, however this is computed again by M. Coornaert in \cite{coornaertasymp}.

In Section \ref{homoesec} we write down explicit generating functions for the number of elements in the free group with a given abelianization. These formulas can be expressed as Chebyshev polynomials -- this is so, because the adjacency matrix of the "recognizing automaton" graph has only two non-trivial eigenvalues, and this is special to free groups. It would be interesting to write down formulas of this type for, eg, surface groups, and see what special functions arise. 

The fact that certain variations on Chebyshev polynomials arise as generating functions give previously unknown positivity result on combinations of their coefficients and shows that the functions $T_n(c \cos x)$ and $U_n(c \cos x),$ where $T$ and $U$ are Chebyshev polynomials of first and second kind respectively, and $c>1$ are positive semi-definite in the sense of Bochner. This, and the central limit theorem for the coefficients of ``Symmetrized Chebyshev Polynomials" appear in the author's paper 
\cite{cheb}.

The Central Limit theorem for distribution of elements of the free group $F_n$ is proved in Section 3, but the methods actually go through without much change to prove a ``Local Limit Theorem''. Such  a theorem was also shown by R. Sharp, using much more heavy lifting in his paper \cite{localsharp}. The central limit theorem was reproved, together with some variants of results of Phillips-Sarnak, Adachi-Sunada, and Katsuda-Sunada in Petridis and Risager's papers \cite{PetridisRisager2006,PetridisRisager2008}. The methods of \cite{PetridisRisager2006} involve perturbation theory, and so are similar to those of the current paper. Results of \cite{PetridisRisager2006} are closely related to those of \cite{krss} -- in that paper we show (using the ergodicity of the  $\SL(n, \integers)$ action on $\reals^n$ and the Central Limit Theorem for free groups in the current paper) that some probabilistic phenomena in the free group $F_n$ can be studied by descending to the abelian quotient.

The Central Limit Theorem has been extended in other ways as well: D. Calegari and Koji Fujiwara proved a central limit theorem for the values of \emph{bicombable} functions on word-hyperbolic groups in \cite{CalefujiCombable}, using Markov chain methods, while M. Horsham and R. Sharp extended the results to \emph{quasi-morphisms} of free groups by using the usual symbolic dynamics and thermodynamic formalism in \cite{HorshamSharpGroups}. 

Lest one think that every function of interest on free (or word-hyperbolic) group satisfies  a central limit theorem, we should note the results of Guivarc'h-LeJan (\cite{GuiLeJanAsymp,GuivLeJanAsympRect} )and Vardi (\cite{vardidedekind}), which show that the the distribution of lengths of geodesics on the modular surface satisfies a \emph{stable law} of Cauchy type.

\subsection*{Then what happened? Walks on graphs}
In Sections \ref{limitp} and \ref{bias} we look at homology modulo a prime $p$ and derive the expected equidistribution results (and also the analogue of \emph{Chebyshev bias}, see \cite{sarnakrubinstein}, which in this case is completely explicit).  More importantly, however, a study of the argument showed that instead of a finite abelian group we can take any \emph{compact} (in particular, any \emph{finite}) group -- the harmonic analysis goes through, although with some more work. The arguments in this paper are a little sketchy, but are presented in full detail in my papers \cite{riirred,effbounds}. These papers, together with \cite{rivzdense} are devoted to proving that certain phenomena in algebraic groups, as well as ``geometric'' groups, like the mapping class group and the outer automorphism group of the free group (and a large class of subgroups) are generic (which means that in large subsets of the groups in question, the vast majority of elements have a certain property -- see \cite{KapovichSchupp} for other examples). The way the results of the current paper are used is essentially through a ``Chinese remaindering'' argument -- if a certain property does \emph{not} hold for some fraction of the elements in the projection of an algebraic group (scheme) over $\integers/ p\integers,$ then it does not hold generically in the group over $\integers.$ Using property $T$ and a more refined analysis (as in \cite{effbounds}) give estimates of convergence speed. The appearance of the paper \cite{riirred} is responsible the subsequent appearance of E.~Kowalski's book \cite{Kowalskibook}, where these rather simple ideas are couched in a rather formidable apparatus.

\subsection*{Then what happened? Topological entropy} In the mid-to-late 1990s, the spectacular results of G. Besson, G. Courtois, and S. Gallot on ``volume rigidity'' of locally symmetric spaces (see \cite{BCG95}) were generating a lot of excitement. The result was that among all the metrics of a given volume on a hyperbolic manifold, the metric of constant sectional curvature minimizes volume entropy -- this answered a conjecture of Gromov stated in \cite{GromovVol}, and previously known only in dimension two (thanks to A. Katok's result \cite{katokgeodesicsent}). Any time a function has a single minimum, there is a suspicion that some sort of convexity is afoot, and entropy in the simplest setting (see, for example, \cite{shannoninfo}) is a convex function of the probabilities, and this pushed the author to analyze topological entropy for walks on graphs as a function of weights on the vertices in Section \ref{entropy}. The methods are again those of perturbation theory. Later, the result was extended to \emph{edge} weightings by S. Lim in \cite{liment}. Lim does \emph{not} prove convexity, but does write down the unique metric of minimal entropy. A related minimality result is proved by I. Kapovich and T. Nagnibeda in their paper \cite{kapovichnagnibeda} for \emph{regular} graphs (their work has its roots in the study of Outer Space. In a different direction, the convexity of entropy was used by I. Kapovich and myself in \cite{kaprivinnomcshane} to show that there is no analogue to McShane's identity in OuterSpace.

\section*{Introduction}

In this paper we begin by studying certain growth functions of the
free group $F_r$, related to well-studied questions on the growth
functions of geodesics on manifolds.
The free group is a relatively simple combinatorial
object, and this allows us to get fairly complete answers to our
questions. Our techniques, which are quite elementary, allow us to get
precise results on the distribution of elements in $F_r$ as a function
of their abelianization and in terms of their abelianization mod
$p$. Our techniques turn out to be easily extensible to the study of
paths in graphs with coefficients in compact groups.

Here is an outline of the paper:
In Section \ref{modsec} we set up an equivalence between counting
cyclically reduced words on the free group $F_r$ and counting circuits
on an associated graph $\mathcal{ G}_r,$ which, in turn, involves
understanding the spectrum of the adjacency matrix of $\mathcal{ G}_r$ (of
course the answer is easily obtained, and is well-known; for
convenience we state it as Theorem \ref{cycredno}). 
We use this framework to obtain a generating function for the number of
elements of a fixed cyclically reduced length with prescribed
abelianization (or homology class). This turns out to be
essentially a Chebyshev polynomial of the first kind; see Definition
\ref{rfun} of the function $R_r$ and Theorem \ref{homoenum} (a very
brief introduction to Chebyshev polynomials is 
given in Section \ref{chebint}). The fact that the function 
$R_r(c; {\bf x})$ (at least for some special values of the parameter
$c$) is a combinatorial generating function implies a
previously unnoticed positivity result on Chebyshev polynomials; this result
is generalized in Section \ref{genanal} in Theorems \ref{nonneg}
and \ref{mnonneg}. Theorem \ref{homoenum} is used
in Section \ref{limitu} to derive a limiting
distribution (as $n$ tends to infinity) of cyclically reduced words
length $n$ among the possible homology classes. From the analytic
standpoint this is also a qualitative result about Chebyshev
polynomials, complementing the positivity Theorems \ref{nonneg} and
\ref{mnonneg}. In Section \ref{limitp} we show that if we study
homology $\mod p$, then the cyclically reduced words in $F_r$ are
asymptotically equidistributed among the $p^r$ classes in $H_1(F_r,
\integers/p \integers).$ We also succeed in estimating the extent to
which the cyclically reduced words in $F_r$ are {\em not}
equidistributed mod $p$ (Section \ref{bias}).

While the results in Sections \ref{limitu} and \ref{limitp} seem to
depend on the explicit generating function that we have obtained, in
Section \ref{graphs} we show that our techniques are more general,
and use them to study the equidistribution properties of long walks
on regular graphs -- we obtain a complete answer (Theorem
\ref{walks}) -- and, without any change, closed orbits of irreducible
primitive Markov processes (with a finite number of states). The
arguments use elementary perturbation theory and the necessary
technical results are contained in Section \ref{perturb}. 

In Section \ref{noback} we extend our methods to study the functions
defined on the \emph{edges} of a graph, and as an application we
derive the statistical properties of long walks \emph{without
backtracking} on the edges of an undirected graph.

We apply our
methods to derive equidistribution results for long walks with
coefficients in compact groups in Sections \ref{primedist} and
\ref{compgp}. Our results are completely explicit, in that knowing the
irreducible representations of the group in question allows us to
obtain complete asymptotics for the convergence to uniformity.
Our results also apply, via the construction of a directed edge graph
to the statistics of ``geodesic'', that is, backtrackless paths
(Section \ref{noback}). This,
in turn, implies a result on the statistical properties of ``primitive''
orbits of Markov processes as above.

In Section \ref{appsec} we point out real and philosophical
applications of the above mentioned result to group theory (where this
all started) and geometry. 

Finally, in Sections \ref{gf}-\ref{gf3} we derive a relationship between the
number of cyclically reduced words and the number of conjugacy classes
of bounded length. While the generating function of the first is a
rational function, the generating function of the second is
the integral of a Lambert series with an infinite number of poles. These
results are then extended to a slightly more general case than that of
free groups. We then (in Section \ref{ihara}) compute a zeta function
for primitive conjugacy classes, and show that this {\em is} a
rational function. 

\section{A model and a generating function}
\label{modsec}

Let $G$ be the free group $F_r=\langle a_1, \ldots, a_r \rangle$, and
let $g \in G$ be an element. The defining property of $G$ is that $g$
is {\em uniquely} represented by a {\em reduced} word in $a_1, \ldots,
a_r$, that is, a word where $a_i$ is never adjacent to $a_i^{-1}$
(Notation: in the sequel we shall write $W$ for $w^{-1}$). We
observe that such words over the alphabet $a_1, A_1, \ldots, a_n, A_n$
are, in turn, be generated by walks on the graph $\mathcal{
G}_r$, constructed as follows: $\mathcal{ G}_n$ has $2 r$ vertices,
labelled with the symbols $a_1, \ldots, a_r, 
A_r, \ldots, A_1$ -- this peculiar order will simplify
notation later. The vertex corresponding to $a_i$ is
connected by an edge to every vertex {\em except} $A_i$. In
particular, there is a loop joining $a_i$ to itself (so that $\mathcal{
G}_r$
 is not a {\em simple} graph). A walk $v_1 v_2 \ldots v_k$ gives the word $v_2 \ldots
v_k$, so the correspondence between walks and words is a $2r-1$-to-$1$
mapping. Note, however, that if we restrict our attention to closed
walks (circuits with basepoint) on $\mathcal{ G}_r$, then those are in bijective
correspondence with {\em cyclically reduced} words in $G$. In
the sequel we will be interested exclusively with cyclically reduced
words. 

\subsection{Counting cyclically reduced words}

To count cyclically reduced words, then, we need to count circuits in
$\mathcal{ G}_r$. This is a well-understood problem:  If $\mathcal{ A}_r$ is the
adjacency matrix of $\mathcal{ G}_r$, then the number of circuits of
length $k$ is equal to the trace of $\mathcal{ A}_r^k$. To compute this trace we
must compute the spectrum of $\mathcal{ A}_r$, and to do this, it is better to
write $\mathcal{ A}_r = J_{2r} - P_r$, where $J_N$ is an $N\times N$
matrix all of whose elements are $1$ and $P_r$ is the $2r \times 2r$
matrix such that
\[(P_r)_{ij} = \begin{cases}
                1, &\text{if $i+j=2r;$}\\
                      0, &\text{otherwise.}
                \end{cases}
\]
In order to compute the spectrum of $\mathcal{ A}_r$, we note first that the
matrix $J_{2r}$ has rank $1$. The kernel of $J_{2r}$ is 
\[\ker J_{2r} = \{(v_1, \ldots, v_{2r}) \bigm| \sum_{i=1}^{2 r} v_i =
0\},\] while the vector ${\bf 1} = (1, \ldots, 1)$ is the eigenvector
of eigenvalue $2 r$. 

The spectrum of $P_r$ is not much more difficult to compute: The vector
${\bf 1}$ is the eigenvector of $P_r$ as well as of $J_{2r}$, this
time with eigenvalue $1$. To compute the rest of the spectral
decomposition, let ${\bf x}$ be an eigenvector of $P_r$ orthogonal to
${\bf 1}$, and let $\lambda$ be the corresponding eigenvalue. Then we
have the following set of equations:
\begin{gather*}
\sum_{j=1}^{2r} x_j = 0 \label{ort}\\
x_j = \lambda x_{2r - j+1}, \qquad j=1, \ldots 2 r.
\end{gather*}
Since at least one of the $x_j$ is not equal to zero, we see that
$\lambda^2 = 1$, so $\lambda = \pm 1.$ The orthogonality condition
Eq. (\ref{ort}) can be rewritten as $\sum_{j=1}^r (1+\lambda) x_j = 0.$
Suppose $\lambda = -1$. Then, 
Eq. (\ref{ort}) holds {\em a forteriori}, and so the
eigenspace of of $-1$ is $r$-dimensional. On the other hand, if
$\lambda = 1$, then we have the additional constraint that
$\sum_{j=1}^r x_j=0,$ so the eigenspace of $1$ is $n-1$ dimensional. 
Putting this all together, we see that the spectrum of the adjacency
matrix $\mathcal{ A}_r$ is $(2r-1, \underbrace{1, \ldots, 1}_r, \underbrace{-1,
\ldots, -1}_{r-1}).$ We see therefore:

\begin{theorem}
\label{cycredno}
The number of cyclically reduced words of length $m$ in $F_r$ is equal
to $(2r-1)^m + 1 + (r-1)[1+(-1)^m].$
\end{theorem}

\section{Counting cyclically reduced words in homology classes}
\label{homoesec}

Recall that the abelianization of $F_r$ is $\integers^r$, generated by the
classes of $[a_1], \ldots, [a_r]$ of $a_1, \ldots, a_r$
respectively. To compute the homology class of a word 
$w$ in $F_r$ we simply count the total exponents $e_1(w), \ldots, e_r(w)$ of
the generators used to write $w$. Then, $[w] = e_1(w) [a_1] + \cdots +
e_r(w) [a_r]$. In this section we will compute the following
generating function:

\[\mathcal{ H}_r^{(k)}(x_1, \ldots, x_r) = \sum_{w\in W_k} \prod_{i=1}^r x_i^{e_i(w)},\]
where the sum is taken over the set $W_k$ of all cyclically reduced words $w$ in $a_1,
\ldots, a_r, A_1, \ldots, A_r$ of length $k$.

To compute $\mathcal{ H}_r^{(k)}$, we return to circuits in $\mathcal{
G}_r$. Given a circuit $c=v_1, \ldots, v_k, v_{k+1} = v_1$, the
contribution of $c$ to $\mathcal{ H}_r^{(k)}$ is the monomial $m_c$ given by the
following iterative procedure: we start with $1$, every time we see
the vertex $a_i$, we multiply $m_c$ by $x_i$, and every time we see
$A_i$, we multiply $m_c$ by $1/x_i$. From this, it follows that:

\thm{theorem}
{
\label{enumer}
The Laurent polynomial $\mathcal{ H}_r^{(k)}$ is given by $\tr B_r^k$,
where $B_r =  D_r \mathcal{ A}_r$, where, in turn,  
\[D_r = \begin{pmatrix}
x_1 &        &      &            &        & \cr
                     & \ddots &     &            &        & \cr
                     &        & x_n &            &        & \cr
                     &        &     & 1/x_n&      & \cr
                     &        &     &            & \ddots & \cr
                     &        &     &            &        &
                     1/x_1
\end{pmatrix}\]
}

Computing the trace of $B_r^k$ seems daunting at first, but one can
use the approach we have used to prove Theorem \ref{cycredno}. 

First, note that
\[B_r = D_r \mathcal{ A}_r = D_r J_{2r} - D_r P_r.\]
Evidently, the rank of $D_r J_{2r}$ is still equal to $1$, and 
\[\ker D_r J_{2r} = \{ {\bf v} = (v_1, \ldots, v_{2r}) \bigm| 
                    \sum_{j=1}^{2r} v_j = 0\}\]

Note further that an eigenvector ${\bf v}$ of $D_r P_r$, such that
${\bf v} \in \ker D_r J_{2r}$, with associated eigenvalue $\lambda$, is
also an eigenvector of $B_r$, with associated eigenvalue
$-\lambda$. To find such an eigenvector, we must solve the system of
equations: 

\begin{gather*}
\sum_{j=1}^{2r} v_j  = 0 \\
\lambda v_j= v_{2r-j+1}/x_j,\quad j\leq r\\
\lambda v_j= v_{2r-j+1} x_j,\quad j > r.
\end{gather*}
We find, as before, that $\lambda=\pm 1$. The first equation reduces
 (almost as before) to 
\[\sum_{j=1}^r v_j(1+\lambda x_j) = 0,\] so that the eigenspaces of
both $1$ and $-1$ are $(r-1)$-dimensional. What are the two remaining
eigenvalues $\mu_1$ and $\mu_2$ of $B_r$? Note that since $\det D_r =
1,$ we know that  $\det B_r = \det \mathcal{ A}_r.$ Note now that 
$\det B_r = \mu_1 \mu_2 (-1)^{r-1}$, while $\det \mathcal{ A}_r =
(2r-1)(-1)^{r-1}.$ So
\begin{equation}
\mu_1 \mu_2 = 2r -1 \label{prodeq}.
\end{equation}
On the other hand, 
\begin{equation}
\mu_1 + \mu_2 = \tr B_r = \sum_{j=1}^r (x_j + {\frac1{x_j}}). \label{sumeq}
\end{equation}
Denoting $y_r = {\frac12} \sum_{j=1}^n (x_j + 1/x_j),$
we see that $\mu_1, \mu_2$ are the two roots of the equation 
$z^2 - 2 y_r z + (2r-1)=0,$ so that:
\begin{align*}
\mu_1& = y_r - \sqrt{y_r^2 - (2r-1)},\\
\mu_2& = y_r + \sqrt{y_r^2 - (2r-1)}.
\end{align*}
The trace of $B_r^k$ is then equal to $\mu_1^k + \mu_2^k + (r-1)[1 +
(-1)^k].$ This can be expressed in terms of well known special
functions, if we make the substitution $y_r = \sqrt{2r-1}
y_r^\prime.$ Then, 
\begin{eqnarray*}
\mu_1^k&= (2r-1)^{k/2} \left(y_r^\prime - \sqrt{{y_r^\prime}^2 - 1}\right)^k,\\
\mu_2^k&= (2r-1)^{k/2} \left(y_r^\prime + \sqrt{{y_r^\prime}^2 - 1}\right)^k,
\end{eqnarray*}
and so 
\begin{eqnarray*}
\mu_1^k + \mu_2^k &= (2r-1)^{k/2} \left\{\left(y_r^\prime -
\sqrt{{y_r^\prime}^2 - 1}\right)^k +\left(y_r^\prime +
\sqrt{{y_r^\prime}^2 - 1}\right)^k\right\}\\
 &= 2 (\sqrt{2r-1})^k T_k(y_r^\prime),
\end{eqnarray*}
where $T_k(x)$ is the $k$-th Chebyshev polynomial of the first kind.
To simplify notation in the sequel, we define:

\begin{definition}
\label{rfun}
\begin{eqnarray*}
R_n(c; x_1, \dots, x_k) &= T_n\left({\frac{c} {2k}}\sum_{i=1}^k\left(x_i +
{\frac1{x_i}}\right)\right)\\
S_n(c; x_1, \dots, x_k) &= U_n\left({\frac{c}{2k}}\sum_{i=1}^k\left(x_i +
{\frac1{x_i}}\right)\right).
\end{eqnarray*}
\end{definition}

And to summarize:
\begin{theorem}
\label{homoenum}
The number of cyclically reduced words of length $k$ in $F_r$
homologous to $e_1 [a_1] + \cdots + e_r [a_r]$ is equal to the
coefficient of $x_1^{e_1} \cdots x_r^{e_r}$ in 
\begin{equation}
\label{genfn}
2\left(\sqrt{2r-1}\right)^k 
R_k({\frac{r} {\sqrt{2r-1}}}; x_1, \dots, x_r)  + (r-1)[1 + (-1)^k]
\end{equation}
\end{theorem}

\begin{remark}
The rescaled Chebyshev polynomial $T_k(a x)/a^k$ is called
the $k$-th Dickson polynomial $T_k(x, a)$ (see \cite{SchurCheb}).
\end{remark}

\section{Some facts about Chebyshev polynomials}
\label{chebint}

The literature on Chebyshev polynomials is enormous; \cite{rivlin} is
a good to start. Here, we shall supply the barest essentials in an
effort to keep this paper self-contained. 

There are a number of ways to define Chebyshev polynomials (almost as
many as there are of spelling their inventor's name). A standard
definition of the {\em Chebyshev polynomial of the first kind}
$T_n(x)$ is:

\begin{equation}
\label{def1}
T_n(x) = \cos n \arccos x.
\end{equation}
In particular, $T_0(x) = 1,$ $T_1(x) = x.$ Using the identity 
\begin{equation}
\label{cosid}
\cos(x+y) + \cos(x-y) = 2 \cos x \cos y
\end{equation}
we immediately find the three-term recurrence for Chebyshev
polynomials:
\begin{equation}
\label{threerec}
T_{n+1}(x) = 2 x T_n(x) - T_{n-1}(x).
\end{equation}
The definition of Eq. (\ref{def1}) can be used to give a ``closed
form'' used in Section \ref{homoesec}:
\begin{equation}
\label{sqrtdef}
T_n(x) = {\frac12}\left[\left(x - \sqrt{x^2-1}\right)^n + \left(x +
\sqrt{x^2-1}\right)^n\right].
\end{equation}
Indeed, let $x = \cos \theta.$ then $\left(x - \sqrt{x^2-1}\right)^n =
\exp(-i n \theta),$ while $\left(x + \sqrt{x^2-1}\right)^n = \exp(i n
\theta),$ so
${\frac12}\left(x - \sqrt{x^2-1}\right)^n + \left(x +
\sqrt{x^2-1}\right)^n = \Re \exp(i n \theta) = \cos n \theta.$

Though we will not have too many occasions to use them, we also define
Chebyshev polynomials of the second kind $U_n(x)$, which can again be
defined in a number of ways, one of which is:
\begin{equation}
\label{derivdef}
U_n(x) = {\frac1{n+1}}T_{n+1}^\prime(x).
\end{equation}
A simple manipulation shows that if we set $x = \cos \theta,$ as
before, then
\begin{equation}
\label{trigdef}
U_n(x) = \frac{\sin (n+1) \theta}{\sin \theta}.
\end{equation}
In some ways, Schur's notation $\mathcal{ U}_n = U_{n-1}$ is
preferable. In any case, we have $U_0(x) = 1$, $U_1(x) = 2 x,$ and
otherwise the $U_n$ satisfy the same recurrence as the $T_n$, to wit,
\begin{equation}
\label{trec2}
U_{n+1}(x) = 2 x U_n(x) - U_{n-1}(x).
\end{equation}
From the recurrences, it is clear that for $f=T, U$,
$f_n(-x) = (-1)^n f(x),$ or, in other words, every second coefficient
of $T_n(x)$ and $U_n(x)$ vanishes. The remaining coefficients
alternate in sign; here is the explicit formula for the coefficient
$c_{n-2m}^{(n)}$ of
$x^{n-2m}$ of $T_n(x):$
\begin{equation}
\label{coefform}
c_{n-2m}^{(n)} = (-1)^m {\frac{n}{n-m}} {\binom{n-m}{m}} 2^{n-2m-1}, \qquad
m=0, 1, \ldots, \left[{\frac{n}2}\right].
\end{equation}
This can be proved easily using Eq. (\ref{threerec}).

\section{Analysis of the functions $R_n$ and $S_n$.}
\label{genanal}

In view of the alternation of the coefficients, the appearance of the
Chebyshev polynomials as generating functions in Section
\ref{homoesec} seems a bit surprising, since combinatorial generating
functions have non-negative coefficients. Below we state and prove a
generalization. Remarkably,
Theorems \ref{nonneg} and \ref{mnonneg} do not seem to have been
previously noted. 

\begin{theorem}
\label{nonneg}
Let $c > 1.$ Then all the coefficients of $R_n(c; x)$
are non-negative. Indeed the coefficients of $x^n, x^{n-2}, \ldots,
x^{-n + 2}, x^{-n}$ are positive, while the other coefficients are
zero. The same is true of $S_n$ in place of $R_n.$
\end{theorem}

\begin{proof} Let $a_n^k$ be the coefficient of $x^k$ in
$U_n((c/2)(x+1/x)).$ The recurrence gives the following recurrence for the
$a_n^k:$
\begin{equation}
\label{newrec}
a_{n+1}^k = c(a_n^{k-1}+a_n^{k+1}) - a_{n-1}^k.
\end{equation}
Now we shall show that the following always holds:

\begin{description}
\item[(a)] $a_n^k \geq 0$ (inequality being strict if and only if
$n-k$ is even). 

\item[(b)] $a_n^k \geq \max(a_{n-1}^{k-1}, a_{n-1}^{k+1}),$ the
inequality strict, again, if and only if $n-k$ is even.

\item[(c)] $a_n^k \geq a_{n-2}^k$ (strictness as above).
\end{description}

The proof proceeds routinely by induction; first the induction step
(we assume throughout that $n-k$ is even; all the quantities involved
are obviously $0$ otherwise):

By induction $a_{n-1}^k < \min(a_n^{k-1}, a_n^{k+1}),$ so 
by the recurrence \ref{newrec} it follows that 
$a_{n+1}^k > \max(a_n^{k-1}, a_n^{k+1}).$ (a) and (c) follow
immediately. 

For the base case, we note that $a_0^0 = 1,$ while $a_1^1 = a_1^{-1} =
c > 1,$ and so the result for $U_n$ follows. Notice that the above
proof does {\em not} work for $T_n$, since the base case
fails. Indeed, if $b_n^k$ is the coefficient of $x^k$ in
$T_n((c/2)(x+1/x))$, then $b_0^0 = 1$, while $b_1^1 = c/2$, not
necessarily bigger than one. However, we can use the result for $U_n$,
together with the observation (which follows easily from the addition
formula for $\sin$) that 
\begin{equation}
\label{moretrig}
T_n(x) = {\frac{U_n(x) - U_{n-2}(x)}2}.
\end{equation}
Eq. (\ref{moretrig}) implies that $b_n^k = a_n^k - a_n^{k-2} > 0$, by
(c) above.
\end{proof}

The proof above goes through almost verbatim to show:

\begin{theorem}
\label{mnonneg}
Let $c > 1.$ Then all the coefficients of 
$R_n$ are non-negative. The same is true of $S_n$ in place of $R_n$
\end{theorem}

To complete the picture, we note that:

\begin{theorem}
\label{trivthm}
\[
R_n(1; x) =
{\frac12}\left(x^n +
{\frac1 {x^n}}\right).
\]
\end{theorem}

\begin{proof} Let $x = \exp i\theta.$ Then $1/2(x+1/x) = \cos \theta,$
and $R_n(1; x)=T_n(1/2(x+1/x)) = \cos n \theta = 1/2(x^n + 1/x^n).$ 
\end{proof}

\begin{remark}
For $c<-1$ it is true that all the coefficients of $R_n(c; .)$ and
$S_n(c; .)$ have the same sign, but the sign is $(-1)^n.$ For $|c|<1,$
the result is completely false. For $c$ imaginary, the result is
true. I am not sure what happens for general complex $c$.
\end{remark}

By the formula (\ref{coefform}), we can write
\begin{equation}
\label{uniexp}
T_n\left({\frac{c} 2}\left(x+{\frac{1}{x}}\right)\right) = 
{\frac12} \sum_{m=0}^{\left[{\frac{n}{2}}\right]} (-1)^m {\frac{n}{n-m}}
{\binom{n-m}{m}} c^{n-2m} \left(x+{\frac{1}{x}}\right)^{n - 2m}.
\end{equation}
Noting that 
\begin{equation}
\label{binom}
\left(x+{\frac{1}{x}}\right)^k = \sum_{i=0}^k {\binom{k} { i}} x^{k - 2i}
\end{equation}
we obtain the expansion
\begin{equation}
\label{fullform}
R_n(c; x) = 
c^n \sum_{k=-n}^n x^k \sum_{m=0}^{\left[{\frac{n}{2}}\right]}
\left(-{\frac{1}{c^2}}\right)^m {\frac{n}{n-m}}
{\binom{n-m}{m}} {\binom{n-2m}{(n-2m-k)/2}},
\end{equation}
where it is understood that $\binom{a} { b}$ is $0$ if $b<0,$ or $b > a$,
or $b \notin \integers.$ We shall denote the coefficient of $x^k$ by $t(n, k, c).$

\section{Limiting distribution of coefficients}
\label{limitu} 

While the formula (\ref{fullform}) is completely explicit, and a similar
(though somewhat more cumbersome) expression could be obtained for 
$R_n(c; x_1, \dots, x_k),$ for many purposes it is more useful to have
a limiting distribution formula as given by Theorem \ref{centlim}
below. To set up the framework, we note that since all the
coefficients of $R_n(c; x_1, \dots, x_k)$ are non-negative (according
to Theorem \ref{mnonneg}), they can be thought of defining a
probability distribution on the integer lattice $\integers^k,$ defined by
$p(l_1, \dots, l_k)=[x_1^{l_1}x_2^{l_2}\cdots x_k^{l_k}]R_n(c; x_1,
\dots, x_k)/R_n(c; 1, \dots, 1)$ (where the square brackets mean that we
are extracting the coefficients of the bracketed monomial). Call the
resulting probability distribution $\mathcal{ P}_n(c; {\bf z}),$ where
${\bf z}$ now denotes a $k$-dimensional vector.

\begin{theorem}
\label{centlim}
When $c>1$, the probability distributions
$\mathcal{ P}_n(c; {\bf z}/\sqrt{n})$ converge to a normal distribution on
$\reals^k$, whose mean is ${\bf 0}$, and whose covariance matrix $C$ is
diagonal, with entries  
\[
\sigma^2 = {\frac{c}{k}}\left[1+ \left({\frac{c+1}{c-1}}\right)^{1/2}\right].
\]
\end{theorem}

To prove Theorem \ref{centlim} we will use the method of characteristic
functions (Fourier transforms), and
more specifically at first the {\em Continuity Theorem}
(\cite[Chapter XV.3, Theorem 2]{feller2}), 
\begin{theorem}
\label{ct}
In order that a sequence $\{F_n\}$ of probability distributions
converges properly to a probability distribution $F$, it is necessary
and sufficient that the sequence $\{\phi_n\}$ of their characteristic
functions converges pointwise to a limit $\phi$, and that $\phi$ is
continuous in some neighborhood of the origin. 

In this case $\phi$ is the characteristic function of $F$. (Hence
$\phi$ is continuous everywhere and the convergence $\phi_n\rightarrow
\phi$ is uniform on compact sets).
\end{theorem}

The characteristic function $\phi_n$ of 
$\mathcal{ P}_n(c; {\bf z})$ is simply \[
R_n(c; \exp( i \theta_1), \ldots,
\exp(i \theta_k))/R_n(c; 1, \ldots, 1),
\]
since the characteristic function is just the generating function evaluated on the unit circle.

 By definition of $R_n$, 
\begin{gather*}
R_n(c; \exp( i \theta_1), \ldots, \exp(i
\theta_k)) = T_n\left({\frac{c}{k}}\sum_{j=1}^k \cos \theta_j\right),\\
R_n(c; 1, \ldots, 1)) = T_n\left({\frac{c}{k}}\sum_{j=1}^k \cos 0\right)
= T_n(c).
\end{gather*}

We now use the form of Eq. (\ref{sqrtdef}):

\[
T_n(x) = {\frac12}\left(\left(x - \sqrt{x^2-1}\right)^n + \left(x +
\sqrt{x^2-1}\right)^n\right),
\]
setting 
\[
u = \sum_{j=1}^k \cos {\frac{\theta_j} {\sqrt{n}}},\qquad\btheta = (\theta_1, \dots, \theta_k),\]
we get 

\begin{equation}
\label{phiform}
\phi_n({\btheta/\sqrt{n}}) = 
{\frac1{T_n(c)}}\left\{ 
{\frac12}\left({\frac{c}{k}} u +
\sqrt{{\frac{c^2}{k^2}} u^2-1}\right)^n +
{\frac12}\left({\frac{c}{k}} u - 
\sqrt{{\frac{c^2}{k^2}} u^2-1}\right)^n\right\}.
\end{equation}
Notice, however, that for $c>1$, the ratio of the second term in
braces to the first is exponentially small as $n\rightarrow \infty$,
since the first term grows like $(c+\sqrt{c^2-1})^n$, while the second
as $(c-\sqrt{c^2-1})^n$ (since $\cos {\frac{\theta_j}{\sqrt{n}}}
\rightarrow 1$). Since, for the same reason, $2T_n(c) =
(c+\sqrt{c^2-1})^n[1 + o(1)],$ we can write:
\[\phi_n({\frac{\btheta}{\sqrt{n}}}) = 
\left[\frac{{\frac{c}{k}} u +
\sqrt{{\frac{c^2}{k^2}} u^2-1}}{c+\sqrt{c^2-1}}\right]^n + o(1).
\]
Substituting the Taylor expansions for the cosine terms (hidden in $u$
for typesetting reasons), we get:
\begin{equation}
u = k + {\frac1{2n}} \langle\btheta, \btheta\rangle +
o(1/n),
\end{equation}
so
\begin{equation}
\label{lintrm}
{\frac{c}{k}} u = c + {\frac{c}{2kn}} \langle \btheta, 
\btheta\rangle + o(1/n).
\end{equation}
A similar computation gives
\begin{equation}
{\frac{c^2}{k^2}} u^2 = c^2 + {\frac{c^2}{k n}} \langle \btheta,
\btheta\rangle + o(1/n).
\end{equation}
Substituting the last expansion into the square root, we see that 
\begin{eqnarray*}
\sqrt{{\frac{c^2}{k^2}} u^2 - 1} &= \sqrt{c^2-1}\sqrt{1+
{\frac1n}\left[{\frac{c^2} {(c^2-1) k}} \langle\btheta,
\btheta\rangle + o({\frac1{n}})\right]}\\ 
&= 
\sqrt{c^2-1} \left[1+{\frac1{2n}}{\frac{c^2} {(c^2-1) k}}
\langle\btheta, \btheta\rangle \right] + o({\frac1{n}}). 
\end{eqnarray*}
Adding Eq. (\ref{lintrm}) and collecting terms, get
\begin{equation}
\frac{\frac{c}{k} u +
\sqrt{{\frac{c^2}{k^2}} u^2-1}}{c+\sqrt{c^2-1}} =
1+{\frac1{2 n}} \left(1+{\frac1{c+
\sqrt{c^2-1}}}\right) \left({\frac{c}{k}} + {\frac{c^2}{(c^2-1)^{1/2}
k}}\right) \langle \btheta, \btheta\rangle + o({\frac1{n}}).
\end{equation}
Performing some further simplifications, we see that
\[
\phi_n({\frac{\btheta}{\sqrt{n}}}) = 
\exp\left(-\frac12\btheta^t C \btheta\right) + o(1),
\]
where $C$ is the covariance matrix described in the statement of
Theorem \ref{centlim}, and Theorem \ref{centlim} follows immediately.

%\thm{remark}
%{
%{\rm 
%Theorem \ref{centlim} talks about convergence in the sense of
%distribution, however an easy consequence (see \cite[Chapter
%XV.5, Theorem 3]{feller}) is that there the densities converge as well, so
%that $p_n(c; {\bf z}) - \mathcal{ N}_C(x/\sqrt{n}) = o(1/\sqrt{n})$
%uniformly in $x$, where $p_n$ is the probability density defined in
%the beginning of this section and $\mathcal{ N}_C$ is the normal density with
%covariance matrix $C$.
%}
%}

\begin{remark}
The speed of convergence in Theorem \ref{centlim} can be estimated
using standard technology (see \cite[Chapter XVI]{feller2},
\cite[Chapter III.11]{shiryaev}), but the speed of convergence in practice
(as checked by numerical experiments) seems to be much better than the general
estimates. Indeed the $L^1$ difference between $\mathcal{ P}_n$ and the
normal distribution appears to decrease almost exactly linearly in
$n$. 
\end{remark}

\section{Distribution mod $p$}
\label{limitp}
The explicit generating functions derived above can be used to study
the distribution of cyclically reduced words in $F_r$ with respect to
their $\mod p$-homology class (this is the analogue, in this setting,
of the work of \cite{philsarnhomology}).

\begin{theorem}
\label{modp}
Let $h_1$ and $h_2$ be two elements of $H_1(F_r, \integers/p
\integers) = {\integers/p \integers}^r,$ and let $W_{r, n, h_1}$ and 
$W_{r, n, h_2}$ be the numbers of cyclically reduced words in $F_r$
homologous to $h_1$ and $h_2$, respectively. Then,
\begin{equation}
\label{prefourier}
\lim_{n\rightarrow \infty} {\frac{W_{r, n, h_2}}{W_{r, n, h_1}}} = 1.
\end{equation}
\end{theorem}

\begin{proof}
By elementary algebra (in one dimension, formula (\ref{ftrans}), the
statement of theorem is equivalent  
to the statement that 
\begin{equation}
\label{fourier}
\lim_{n\rightarrow \infty} {\frac{\phi_n(\btheta)}{\phi_n({\bf 0})}} =
0,
\end{equation}
for $\btheta = (2 n_1 \pi/p, \dots, 2 n_r \pi/p),$ with not all $n_j$
equal to $0 \mod p,$
where $\phi_n$ is the characteristic function defined in the previous
section.

The estimate of Eq. (\ref{fourier}), however, follows immediately from the
explicit formula (\ref{phiform}): indeed, in the current context, 
\[
u(\btheta) = \sum_{j=1}^k \cos (2 n_j \pi /p),
\]
which is strictly smaller than $u({\bf 0}),$ so the ratio of
$\phi_n(\btheta)$ to $\phi_n({\bf 0})$ goes to zero exponentially fast
in $n$.
\end{proof}
\begin{remark}
Another way to see the equivalence of statements \eqref{prefourier} and \eqref{fourier} is though the well-known fact that the Fourier transform is an isometry (of the corresponding $L^2$ spaces). For a probability density to be close to uniform, its Fourier transform has to be close to that of the uniform distribution, which is a delta function centered at the origin, which is precisely the statement we need.
\end{remark}

\subsection{Deviation from uniformity}
\label{bias}
Although the distribution of homology $\mod p$ approaches uniformity,
it turns out that there is a persistent \emph{bias} in favor of
certain homology classes. This is very much akin to the Chebyshev
bias, analyzed in \cite{sarnakrubinstein}. To simplify the discussion we
project one more time: for each cyclically reduced word in $F_r$
homologous to $a_1^{k_1} a_2^{k_2} \dots a_r^{k_r}$ we consider $k_1 +
\dots + k_r \mod p.$ In this case we have a univariate distribution,
whose generating function is given by $\psi_n(x)= R_n(c; x, \dots,
x),$ with $c = \frac{r}{\sqrt{2r - 1}}$ (as per formula (\ref{genfn};
we leave in the general $c$, to underline that our results apply to
general question on distribution of coefficients of the Laurent polynomials
$R_n$). 

The number of elements congruent to $q \mod p$ is given by 
\begin{equation}
\label{ftrans}
\mathcal{ N}_{n,q} = {\frac{1}{p}}\sum_{j=0}^{p-1} \chi^{-qj}\psi_n(\chi^j),
\end{equation}
where $\chi=\exp(2\pi i /p)$ is a primitive $p$-th root of unity. 
Let us recall that
\begin{equation}
\psi_n(e^{i x}) = {\frac1 {T_n(c)}}\left\{ 
{\frac12}\left(c \cos x  +
\sqrt{c^2 \cos^2 x -1}\right)^n +
{\frac12}\left(c \cos x - 
\sqrt{c^2 \cos^2 x -1}\right)^n\right\}.
\end{equation}
Note the following properties of the function $\psi_n$:
\begin{subequations}\label{props}
\begin{gather}
\psi_n(1/x) = \psi_n(x),\label{propa}\\
\mbox{If $c \cos x < 1,$ then $| \psi_n(\exp(i x)) | T_n(c) \leq
1.$}\label{propb}\\
\psi_n\left\{\exp(i (\pi - x))\right\} = (-1)^n
\psi_n\left\{\exp(i x)\right\}\label{propc}\\
\mbox{If $c \cos x \geq 1,$ then $\psi_n(\exp(i x)) > 0.$}\label{propd}\\
\mbox{If $x \in [0, \arccos 1/c], n \gg 1$  then 
${\frac{|\psi_n(\exp(i x))| T_n(c)}  {\left[c + \sqrt{c^2 \cos^2 x -
1}\right]^n}} = 1 + o(1),$}\label{prope}\\
\psi_n(\exp(i x_1)) = o(\psi_n(\exp(i x_2))\mbox{\ for\ }0\leq
x_2 < \arccos 1/c,~x_2 < x_1 < \pi - x_2.\label{propf}
\end{gather}
\end{subequations}

Using Property \eqref{propa}, we can write
\begin{equation}
\label{cossum}
\mathcal{ N}_{n,q} = {\frac{1}{p}}\left[\psi_n(1) + 2
\sum_{j=1}^{{\frac{p-1}{2}}} 
\cos {\frac{2\pi q j}{p}} \psi_n(\chi^j)\right].
\end{equation}

Since $\cos{\frac{2\pi m}{p}}< 1$ is monotonically decreasing as a
function of $m$ for $0\leq m \leq {\frac{p-1}{2}}$, 
we see:

\begin{theorem}
\label{evenbias}
For sufficiently large even $n$, $\mathcal{ N}_{n, q} < \mathcal{ N}_{n, 0}.$
\end{theorem}
\begin{proof}
This is an immediate consequence of the monotonicity of $\cos$,
equation (\ref{cossum}) and Properties \eqref{propa}, (\eqref{propc}, \eqref{propd} and  \eqref{propf} above.
\end{proof}

For $q\neq 0 \mod p$, the term largest in absolute value in the sum
(aside the $\psi_n(1)$ term) on the right hand side of
eq. (\ref{cossum}) is the $\psi_(\chi^{{\frac{p-1}2}})$ term, so if
we assume that $n$ is even, then the next largest (after $\mathcal{ N}_{n,
0}$) term will be $\mathcal{ N}_{n, p-2}$  (since $(p-2)
[(p-1)/2] = 1 \mod p$), then $\mathcal{ N}_{n, p-4}$, and so on. For $n$
odd, the ordering is reversed.

\section{An extension and limiting distributions for graphs}
\label{graphs}

An inspection of the proof of Theorem \ref{centlim} reveals that in
order to show that for a sequence of probability distributions $\{P_n(x)\}$
on $\integers$, the distributions $\{P_n(x/\sqrt{n})\}$ converged to a
limiting normal distribution with mean $0$, we 
used the following conditions (we will state them in a univariate
setting for simplicity; the multivariate case is the same): 

\medskip\noindent
{\bf Condition 1.} The characteristic function of $\{P_n\}$ has the
form
\[
\chi(P_n) = f^n(\theta) + o(1),
\]
where $f_j(\theta)$ is twice continuously differentiable at $0$, so that
$f_j(\theta) = a_j + b_j \theta + c_j\theta^2 + o(\theta^2).$

\medskip\noindent
{\bf Condition 2.}
\[
a_1 = 1,\qquad b_2 = 0,\qquad c_2 < 0.
\]

Suppose now we generalize the setting of Section \ref{modsec} as
follows:

Let $\mathcal{ G}$ be a connected $r$-regular non-bipartite graph,
directed or not, (possibly with self-loops and multiple edges), on $k$
vertices. Let $v_1$ and $v_2$ be two vertices of $\mathcal{ G}$. Consider
now the set $W_N$ of all closed walks (circuits) of length $N$ on
$\mathcal{ G}$. Let ${\bf f}: V(\mathcal{ G}) \rightarrow R$ be a function
assigning a weight to each vertex of $\mathcal{ G},$ and define a random
variable $X_{\bf f}$ to be $\sum_{l=1}^N {\bf f}(v_l)$ for  $w = v_1,
\ldots, v_N \in W_N.$ What can we say about the distribution of
$X_{\bf f}$? It turns out that asymptotically we can say a lot. First,
however, define 
\[\mu({\bf f}) = {\frac1k}\sum_{j=1}^k {\bf f}(v_j),\]
and ${\bf f}_0 = {\bf f} - \mu({\bf f}) {\bf 1}.$
Define further the {\em Laplacian} $\Delta(\mathcal{ G})$ of $\mathcal{ G}$ to
be $\Delta(\mathcal{ G})= r {\bf I} - A(\mathcal{ G}),$ and define
$\Delta_0(\mathcal{ G})$ to be $\Delta(\mathcal{ G})$ viewed as an operator on
the orthogonal complement to ${\bf 1}$ (that is, vectors with $0$
sum). Let $P_N(x)$ be the distribution of $X_{\bf f}$ on $W_N$.

\begin{theorem}
\label{walks}
The distributions $P_N((x-N \mu({\bf f}))/\sqrt{N})$ converge to a
balanced (that is, mean $0$) normal distribution with variance 
\begin{equation}
\label{varfrm}
\sigma^2({\bf f}) = {\frac1k} \left[-\|{\bf f}_0\|^2 
+ 2 r {\bf f}_0^t \Delta_0^{-1}(\mathcal{ G}) {\bf f}_0 \right]
         =
{\frac1k}\left[{\bf f}_0^t (-{\bf I}_0 + 2 r \Delta_0^{-1}(\mathcal{ G}))
{\bf f}_0\right].
\end{equation}
\end{theorem}

\begin{proof}
Exactly as in Section \ref{modsec} we construct a
generating function $g_N$ for $X_{\bf f}$ on $W_N$. To do this, let 
$A$ be the adjacency matrix of $\mathcal{ G}$, and let 
\[
D_k(x) = \begin{pmatrix}x^{{\bf f}(v_1)} &             &          &            &        & \cr
                   &x^{{\bf f}(v_2)}  &          &               &        & \cr
                   &             & x^{{\bf f}(v_3)}         &            &        & \cr
                   &             &          & \ddots     &        & \cr
                   &             &          &            & x^{{\bf f}(v_k)}      &
                
\end{pmatrix}.
\]
Then, 
\[g_N(x) = \tr (D_k(x) A)^N = \sum_{j=1}^k \lambda_j^N(D_k(x)
A),\]
where $\lambda_1, \ldots, \lambda_j$ are eigenvalues, and, just as in
Section \ref{limitu}, we have
$\chi(P_N)(\theta) = g_N(\exp(i \theta))/c_N,$ where 
\[
c_N = \left|W_N\right| = \sum_{j=1}^k \lambda_j^N(A).
\]
Since $\mathcal{ G}$ is an {\em $r$=regular}, {\em non-bipartite} graph, it has
a unique eigenvalue of maximal modulus, and that eigenvalue is
$\lambda_1=r.$

Now, we can directly apply Conditions 1 and 2 (and accompanying
comments) above, and the results of Section \ref{perturb} (noting that
Assumptions 1--4 hold) to obtain 
the desired result (in particular, the estimate needed in Condition 2
is precisely Theorem \ref{posthmsym}). We replaced the resolvent in
formula (\ref{svharm2}) by the equivalent (by the discussion in the
beginning of Section \ref{perturb}) Laplacian form, since that is more
common in graph theory.
\end{proof}
\begin{remark}
\label{varf}
If the vector $\mathbf{f}$ is an eigenvector of $A^t A$ with
eigenvalue $r^2$, the corresponding variance is equal to zero. By
Remark \ref{sharpening} this will not happen, \textit{eg}, if $G$ is a
connected \emph{non-bipartite} \emph{undirected} graph, but it does happen
for general directed graphs; see the discussion of the directed line
graph in Section \ref{noback}. 
\end{remark}

The above remark leads to the following question:

\begin{question}
What combinatorial property of an $r$-regular directed
graph $G$ is reflected in the algebraic statement that the operator
norm of $A_0(G)$ is equal to $r$?
\end{question}
A slight change in notation transforms Theorem \ref{walks} into
a central limit theorem for distributions over closed orbits of
primitive irreducible Markov processes over a finite number of states
-- the irreducibilty is exactly equivalent to the connectivity of the
graph $\mathcal{ G}$ above. For ease of reference we state this as a
separate theorem. The notation for ${\bf f}$, $\mu$, etc, is as
before; the space $W_N$ is now a probability space with the obvious
probability measure; ${\bf P}={\bf P}^t$ is the transition matrix
(note that Remark \ref{varf} remains valid in this setting as well).

\begin{remark}
\label{markov}
Let $P_N(x)$ be the distribution of $X_{\bf f}$ on $W_N$. Then
$P_N((x-N\mu({\bf f}))/\sqrt{N})$ converge to a balanced (that is, mean
$0$) normal distribution with variance 
\begin{equation}
\label{varmark}
\sigma^2({\bf f}) = {\frac1k} \left[-\|{\bf f}_0\|^2 
+ 2 {\bf f}_0^t ({\bf I}_0 - {\bf P}_0)^{-1} {\bf f}_0 \right] =
{\frac1k}\left[{\bf f}_0^t (-{\bf I}_0 + 2 r ({\bf I}_0 - {\bf P}_0)^{-1})
{\bf f}_0\right].
\end{equation}
\end{remark}

\begin{remark}
\label{arbcoeff}
We have actually shown a slightly stronger result: instead of the
trace (distribution over cycles), we could have considered the $ij$-th
element of ${\bf P}$. Since the principal eigenvector varies
continuously under perturbations (see \cite[Chapter II.4.1]{kato95}), we
could have replaced our sample space $W_N$ as above by the space
$\mathcal{ C}_N$ of paths of length $N$ joining the $i$-th to the $j$-th
vertex. An easy computation shows that the covariance is the
covariance  given in equation \ref{varmark}, divided by a further
factor of $k$. The same remark applies to Theorem \ref{walks}.
\end{remark}

\subsection{Distribution modulo a prime}
\label{primedist}
Theorems \ref{walks} and \ref{markov} have particularly simple
analogues if the function $f$ we are studying is integer valued, and
we are interested in the distribution of the $\integers/p
\integers$-valued random variable $Y_f(n)$ which assigns to each cycle
of length $n$ the sum of the values of $f$ modulo $p$. In that case,
under the assumption that the adjacency matrix $A$ (in the context of
Theorem \ref{walks}) or the transition matrix $A$ (in the context of
Theorem \ref{markov}) is irreducible and primitive (the last two $A(\mathcal{ L}_u(G))$
conditions guarantee that $A$ has a single eigenvalue $\lambda_0$ of maximal
modulus, the eigenspace of $\lambda_0$ is one-dimensional, and the
orthogonal subspace is invariant under $A$), then we see
that the distributions $\mathcal{ P}_n$ of $Y_f(n)$ approach the uniform
distribution (on $\integers/p \integers$) exponentially fast in $n$
(though a more reasonable measure of the speed of convergence is the
size of $W_n$, in which case the convergence is polynomial). This
statement follows from the:

\begin{lemma}
If $A$ is a matrix
satisfying the conditions above, then the spectral radius $r_{UA}$ of $U A$,
for $U$ any non-trivial unitary matrix such that the top eigenvector
of $A$ is not also an eigenvector of $U$, is strictly smaller than
that of $A$ ($r_A$).
\end{lemma}

The proof of the lemma is immediate.

In our case, the matrix $U$ is the diagonal
matrix $U(\chi)$ with $u_{jj}=\chi_p^{f_j}$, with $\chi_p$ a non-trivial
$p$-th root of unity. The speed of convergence to the uniform
distribution is given by  $(\max_{\chi^p = 1} r(U(\chi) A))/r(A)$.

\section{Functions on edges and distributions over 
paths without backtracking}
\label{noback}
In this section we consider two kinds of questions, which are seen to
be intimately related. The first is: 

\begin{question}
Let $f$ be a function on the \emph{edges} of $G$. How are the
averages of $f$ over long cycles or paths in $G$ distributed? 
\end{question}
The second question is:

\begin{question}
Let $f$ be a function on the
\emph{vertices} of $G$. How are the averages of $f$ distributed
over long cycles in $G$ \emph{without backtracking} -- such cycles
are more closely related to, \emph{eg}, geodesics on surfaces,
then arbitrary cycles.
\end{question}

Both questions can be answered at the same time by constructing the
\emph{directed line graph} (or \emph{line digraph}) of
$G$. This construction can be performed for either a directed or
undirected graph $G$; In section \ref{lg} we will derive the results
for undirected graphs in detail, whilst in section \ref{dlg} we will
discuss the directed case somewhat more briefly (since the technical
details are essentially identical).

\subsection{The directed line graph of an undirected graph}
\label{lg}

The \emph{directed line graph} of $G$, denoted by $\mathcal{ L}(G)$, is
constructed as follows: The vertices
of $\mathcal{ L}(G)$ are edges of $G$ labelled with a $+$ or a $-$; that
is, to each edge $e$ of $G$ there correspond vertices $e_-$ and $e_+$
of $\mathcal{ L}(G)$. These correspond to the two possible orientations of
$e$: if the vertices of $e$ are $v$ and $w$, then we say that $v$ is
the head of $e_-$, and $w$ the tail (and write $v = h(e_-)$,
$w=t(e_-)$), while for $e_+$ this nomenclature 
is reversed. Two vertices $v_1$ and $v_2$ of $\mathcal{ L}(G)$ are joined
by a (directed) edge if the head of $v_1$ is the same as the tail of
$v_2,$ except that $e_-$ is never joined to $e_+$, and \textit{vice
versa}. We now make some observations and definitions.

\begin{definition}
Let $f$ be a function defined on the vertices of a graph $G$. We say
that a function $g$ defined on the vertices of $\mathcal{ L}(G)$ is the
\textit{gradient} of $f$, and write $g = \nabla f$ if $g(e) = f(h(e))
- f(t(e)).$
\end{definition}

\begin{definition}
We can identify functions on the vertices of $G$ with
(a subset of) functions on the the vertices of $\mathcal{ L}(G)$. To wit, if a
$f$ is a function on the vertices of $G$, we let $\mathcal{ L}f(e) =
f(t(e)).$
\end{definition}

\begin{observation}
\label{btl}
There is a natural correspondence between walks on $\mathcal{ L}(G)$ and walks on
$G$ without backtracking. Indeed, passing through a vertex $e$ of
$\mathcal{ L}(G)$ corresponds to going from $t(e)$ to $h(e)$. Since $e_+$
is not connected to $e_-$ for any $e \in E(G)$, any such walk is
automatically without backtracking. Similarly, a cycle on $\mathcal{ L}(G)$
corresponds to a tailless cycle without backtracking on $G$.
\end{observation}

If $G$ is an $r$-regular graph, then $\mathcal{ L}(G)$ is $r-1$-regular,
in the strong sense: each vertex of $\mathcal{ L}(G)$ has in-degree and
out-degree equal to $r-1$ (thus the total degree is $2r - 2$),
and from the above Observation \ref{btl}, $\mathcal{ L}(G)$ is connected
if and only if $G$ is. It follows that the adjacency matrix $A(\mathcal{
L}(G))$ of $\mathcal{ L}(G)$ is an irreducible nonnegative matrix, all of
whose row and column sums are equal to $r-1$. It follows that the
space of functions on the vertices of $\mathcal{ L}(G)$ orthogonal to the
vector $\mathbf{1}$ is an invariant subspace of $A(\mathcal{ L}(G))$ and
of $A^t(\mathcal{ L}(G))$ -- we will, as before, denote the two
matrices restricted to this subspace by $A_0$ and $A_0^t$,
respectively; the algebraic and geometric multiplicities of the
eigenvalue $r-1$ is equal to $1$, by standard Perron-Frobenius 
theory. Despite this, it turns out that $A^t A$ is spectacularly
degenerate. Indeed, the $ij$-th entry of $A^t A$ is equal to the
number of vertices of $\mathcal{ L}(G)$ adjacent simultaneously to the
$i$-th and the $j$-th vertex. It follows that the $ii$-th entry of
$A^t A$ is equal to $r-1$, while the $ij$-th entry is equal to $r-2$
if the corresponding directed edges of $G$ have the same tail, and is
$0$ otherwise. It follows that 
\begin{equation}
\label{opnorm}
A^t A = I_{2 E(G)} + (r-2) \begin{pmatrix}
J_1 &          &          & \cr
    & J_2       &        &  \cr
    &           & \ddots & \cr
     &          &        &  J_{V(G)}
\end{pmatrix},
\end{equation}
where the last term contains $V(G)$ $r\times r$ blocks, each of which
is the matrix of all $1$s. We thus have the following observation:
\begin{observation}
\label{specobs}
The spectrum of $A^t A$ has the following form: The
eigenvalue $(r-1)^2$ occurs $V(G)$ times, and the corresponding
eigenvectors are given precisely by $\mathcal{ L}f$ for arbitrary functions
$f$ on $G$ (the Perron eigenvector corresponding to the constant
function), while the eigenvalue $1$ occurs $2 E(G) - V(G)$ times. The
eigenvectors are those functions on the directed edges of $G$, for
which, for all vertices $v$ of $G$,  the sum of values on all the
edges \emph{leaving} $v$ is equal to $0$.
\end{observation}

\begin{corollary}
\label{oppnorm}
The operator norm of $A_0$ is equal to $r-1$.
\end{corollary}

Consider now the Laplace operator on $\mathcal{ L}(G)$: $\Delta_{\mathcal{
L}(G)} = (r-1) I - A(\mathcal{ L}(G)).$ We will need the following 
in the sequel:

\begin{theorem}
\label{imdel}
Let $E_{r-1}$ be the eigenspace of $(r-1)^2$ for $A^t A$. 
If $V^*(G)$ is the space of functions on the vertices of $G$, then
\begin{description}
\item[(a)]
\begin{equation*}
E_{r-1}=\mathcal{ L}(V^*(G)),
\end{equation*}

\item[(b)]
\begin{equation*}
\Delta_{\mathcal{ L}(G)}(E_{r-1}) = \nabla(V^*(G)),
\end{equation*}

\item[(c)]
$\nabla(V^*(G)) \cap E_{r-1} \cap \mathbf{1}^\perp =
\emptyset$, unless $G$ is bipartite. 
\end{description}
\end{theorem}

\begin{proof} Part (a) is the content of Observation \ref{specobs}.
Part (b)
is a corollary of Part (a). Indeed, $\Delta_{\mathcal{
L}(G)}(f)(x) = (r-1)f(x) - \sum_{h(x) =t(y)} f(y).$ If $f = \mathcal{
L}g,$ then 
\begin{equation}
\Delta_{\mathcal{ L}(G)}(f)(x) = (r-1)(g(t(x)) - g(h(x))),
\end{equation}
since all the $y$ adjacent to $x$ have the same tail, equal to the
head of $x$.

To show Part(c), suppose $\nabla(V^*(G)) \cap E_{r-1} \neq
\emptyset$. Let $g$ be in the intersection, and $k$ be such that
$\nabla(k) = g.$ It follows that for any $x$, $y$ such that 
$t(x) = t(y)$,
 $g(x)=g(y).$ We see that 
$k(h(x)) - k(t(x)) = k(h(y)) - k(t(y)),$ which implies in turn that
$k(h(x)) = k(h(y))$. So, $k$ is the eigenvector of the $0$ eigenvalue
of the Laplace operator on $G$, and hence is constant, unless $G$ is
bipartite.
\end{proof}

We end this section with a remark necessary to compute distributions,
as done in the following Section \ref{appdist}. To wit:

\begin{remark}
\label{primm}
The adjacency matrix of the line graph of a non-bipartite graph $G$ is
primitive. That is, there is only one eigenvalue on the circle of
radius $r-1$ in the complex plane, and that is $r-1$. Its geometric
multiplicity is $1$.
\end{remark}

\begin{proof}
Doubtlessly there are simpler arguments, but we choose to use the
results (described in \cite{staterI})
on the Ihara zeta function $Z$ of $G$, which can be expressed as a
determinant in two ways: 

The first way (original theorem of Ihara \cite{iharazetaorig}) is:

\begin{equation}
Z^{-1}(u) = (1-u^2)^{\mathcal{R}-1} \det({(1+(r-1)u^2)\mathbf{I} - u A}),
\end{equation}
with $A$ the adjacency matrix of $G$, and $\mathcal{R}$ the rank of
the fundamental group of $G.$

The second way (due to Hyman Bass \cite{bassiharatree} is):

\begin{equation}
Z^{-1}(u) = \det({\mathbf{I} - u M}), 
\end{equation}
where $M$ is the adjacency matrix of the directed line graph of $G$

The equality of the two expressions implies that $v$ is an eigenvalue
of $M$ if and only if 
$v+(r-1)/v$ is an eigenvalue of $A$ (we are ignoring the eigenvalues
$\pm1$, which occur with large multiplicity in the spectrum of
$M$). Suppose that $v$ has modulus $r-1$, so that $v= (r-1)\exp(i
\theta),$ for some $\theta$. It follows that $w=\exp(i \theta) + (r-1)
\exp(-i \theta)$ is an eigenvalue of $A$, and since $A$ is symmetric,
$\theta \in \{0, \pi\}$. If $\theta=0$, $v=r-1$, while if
$\theta=\pi$, $v=-(r-1),$ but then $w=-r$ is an eigenvalue of $A$, and
so $G$ is bipartite.  

The statement about the multiplicity of the eigenvalue $r-1$ is
immediate, since $\mathcal{L}(G)$ is clearly strongly connected.
\end{proof}

We include the following observations both for the sake of
completeness, and in view of Lemma \ref{gradvanish} below.

\begin{lemma}
\label{compform}
\begin{equation*}
\Delta \mathcal{L} = (r-1) \nabla.
\end{equation*}
\end{lemma}

\begin{proof}
Indeed, $\mathcal{L}(f)(x) = f(t(x)).$ Further,
\begin{equation}
\Delta \mathcal{L}(f)(x) = \sum_{t(y) = h(x)} f(t(x) - f(t(y)) =
                        (r-1)(f(t(x)) - f(h(x)) = \nabla(f)(x).
\end{equation}
\end{proof}

\begin{lemma}
\label{innerprod}
For any $f, g \in V^*(G),$ we have
\begin{equation*}
(\mathcal{L} f)^t \nabla{g} = f^t \Delta g.
\end{equation*}
\end{lemma}

\begin{proof}
Indeed, 
\begin{equation}
\begin{split}
(\mathcal{L} f)^t \nabla{g} &= \sum_x (f(t(x)) (g(t(x)) - g(h(x)))\\
                            & =
\sum_{v \in V(G)} \sum_{\text{$w$ adjacent to $v$}} f(v) g(v) - f(v)
g(w)\\
                        & = \sum_{v \in V(G)} f(v) \Delta(g)(v)\\
                        &= f^t \Delta g.\\
\end{split}
\end{equation}
\end{proof}

Consider now a function $g$ on the directed edges of $G$. How do we
decompose it into a gradient and a function orthogonal to gradients?
First, we note that a basis of the gradients is formed by the
gradients of $\delta$ functions:
\begin{equation}
\delta_v(x) = \begin{cases} 1 & x=v,\\
                            0 & \text{otherwise}.\end{cases}
\end{equation}
So that
\begin{equation}
\nabla \delta_v(x) = \begin{cases} 1 & t(x) = v,\\
                                   -1 & h(x) = v,\\
                                   0 & \text{otherwise}.\end{cases}
\end{equation}

The functions $\nabla \delta_v$ form a basis of $\nabla(V^*(G))$,
though not an orthonormal one. Now, note that 

\begin{equation*}
g^t \nabla \delta_v = \sum_{t(x) = v} g(x) - \sum_{h(y)=v} g(y).
\end{equation*}
In other words, 
\begin{lemma}
\label{orthcomp1}
$g$ is orthogonal to the gradients, if and only if the sum of $g$ over
the edges coming into any vertex $v$ is equal to the sum of $g$ over
the edges leaving $v$. An equivalent condition is that $\nabla^t g = 0.$
\end{lemma}

One may ask: what is the orthogonal projection of a given
$\mathcal{L}f$ onto the gradients? The following comes out of an easy
computation: 

\begin{observation}
\label{orthcomp2}
The orthogonal projection of $\mathcal{L}f$ onto the set of gradients
is $\nabla\Delta f.$
\end{observation}

\subsection{Applications to distribution}
\label{appdist}
We can use the results of the previous section to understand the
limiting distribution of functions defined on (directed) edges of
$G$. Indeed, we can use Theorem \ref{walks} in the form corresponding
to Eq. \ref{svharm3} to observe that 
\begin{equation}
\label{linevar}
\sigma^2(\mathbf{f}) = \frac{1}{2rk}\mathbf{f}^t (\Delta_0^{-1})^t((r-1)^2
\mathbf{I} - A(\mathcal{ L}(G))^t A(\mathcal{ L}(G)) \Delta_0^{-1}
\mathbf{f}
%       &= \frac{r-1}{2rk} \left[\mathbf{f}^t \left(\mathbf{I} -
%2(r-1) \Delta_0^{-1}\right) \mathbf{f}\right]\\
\end{equation}
for $\mathbf{f}$ any function on the directed edges of $G$, and
$\Delta_0$ the restriction of the Laplace operator on $\mathcal{ L}(G)$ to
the subspace of $0$-sum vectors.
\begin{lemma}
\label{gradvanish}
The right hand sidef of equation \ref{linevar} vanishes precisely when
$\mathbf{f}$ is the gradient of a function on the vertices of $G$.
\end{lemma}

\begin{proof}
Let $\mathbf{f} = \Delta u.$ By Observation \ref{specobs} we see that
the right hand side of Eq. \ref{linevar} vanishes precisely if $u \in
\mathcal{L}(V^*(G)).$ By part (b) of Theorem \ref{imdel} it follows that this
is so if and only if $\mathbf{f} \in \nabla (V^*(G)).$
\end{proof}

One direction of the above lemma is just common sense, since the sum
over any cycle of a gradient is equal to $0$.

Keeping the above in mind, we note that a simpler form of the
covariance is given by Theorem \ref{walks}:

\begin{equation}
\label{formula1}
\sigma^2(\mathbf{f}) = \frac{r-1}{2rk} \left[\mathbf{f}^t \left(\mathbf{I} -
2(r-1) \Delta_0^{-1}\right) \mathbf{f}\right]
\end{equation}

For functions on the vertices of $G$, the above assumes the form:

\begin{equation}
\label{formula2}
\sigma^2(\mathbf{f}) = \frac{r-1}{2rk} \left[\mathbf{f}^t
\mathcal{L}^t \left(\mathbf{I} - 2(r-1) \Delta_0^{-1}\right)
\mathcal{L} \mathbf{f}\right] 
\end{equation}

\subsection{The line graph of a directed graph}
\label{dlg}

The construction of the line graph of a directed graph $G$ is
essentially the same as that of an undirected graph. This time, the
vertices of $\mathcal{L}(G)$ without labels (so $\mathcal{L}(G)$ has $E(G)$
vertices). The operators $\nabla$ and $\mathcal{L}$ are defined as in 
Section \ref{lg}. We have an observation even simpler than Observation
\ref{btl}: 

\begin{observation}
\label{dbtl}
There is a natural bijective correspondence between walks on $\mathcal{L}(G)$ and walks on $G$.
\end{observation}

If $G$ is an $r$-regular directed graph (by this we mean that both the
in- and out- degree of each vertex is equal to $r$), then so is
$\mathcal{L}(G);$ by Observation \ref{dbtl} $\mathcal{L}(G)$ is
connected whenever 
$G$ is. As before, $A(\mathcal{L}(G))$ is the adjacency matrix of
$\mathcal{L}(G)$. we can compute: 

\begin{equation}
\label{dopnorm}
A^t A = r \begin{pmatrix}
J_1 &          &          & \cr
    & J_2       &        &  \cr
    &           & \ddots & \cr
     &          &        &  J_{V(G)}
\end{pmatrix},
\end{equation}
where each block corresponds to the set of edges of $G$ emanating from
a given vertex. From this we have:

\begin{observation}
\label{dspecobs}
The spectrum of $A^t(\mathcal{L}(G)) A(\mathcal{L}(G))$ has the following
form:
The eigenvalue $r^2$ occurs $V(G)$ times, and the corresponding
eigenvectors are given by $\mathcal{L}f$ for arbitrary functions $f$
on $G$ (The Perron eigenvector corresonding to the constant function)
while the eigenvalue $0$ occurs $E(G)-V(G)$ times. The eigenvectors
are those functions on the edges of $G$ for which the sums of the
values over all edges leaving a vertex $v$ is equal to $0$ (for all
$v$). 
\end{observation}

\begin{corollary}
\label{poppnorm}
The operator norm of $A_0(\mathcal{L}(G))$ is equal to $r$.
\end{corollary}

The Laplace operator on  $\mathcal{ L}(G)$ is defined as: $\Delta_{\mathcal{
L}(G)} = r I - A(\mathcal{ L}(G)).$

We have 

\begin{theorem}
\label{dimdel}
Let $E_r$ be the eigenspace of $r^2$ for $A^t A$. 
If $V^*(G)$ is the space of functions on the vertices of $G$, then
\begin{description}
\item[(a)]
\begin{equation*}
E_r=\mathcal{ L}(V^*(G)),
\end{equation*}

\item[(b)]
\begin{equation*}
\Delta_{\mathcal{ L}(G)}(E_r) = \nabla(V^*(G)),
\end{equation*}
\end{description}
\end{theorem}

We also include

\begin{remark}
\label{dprimm}
The adjacency matrix of the line graph of $G$ is primitive if the
adjacency matrix of $G$ is.
\end{remark}

\begin{proof}
We use Observation \ref{dbtl} and Theorem \ref{zetadet} to note that
the non-zero eigenvalues of $G$ are exactly the same as those of
$\mathcal{L}(G)$, since $\det(I - u A(G)) = \det(I - u
A(\mathcal{L}(G))).$ 
\end{proof}

\begin{lemma}
\label{dcompform}
\begin{equation*}
\Delta \mathcal{L} = r \nabla.
\end{equation*}
\end{lemma}

\begin{lemma}
\label{dinnerprod}
For any $f, g \in V^*(G),$ we have
\begin{equation*}
(\mathcal{L} f)^t \nabla{g} = f^t \Delta g.
\end{equation*}
\end{lemma}

Lemma \ref{orthcomp1} and Observation \ref{orthcomp2} go through
without change.

The results of section \ref{appdist} go through essentially without
change. Since some constants change we restate them here. First, let
$f$ be a function defined on the edges of $\mathcal{L}(G)$. We see
that:

\begin{equation}
\label{dlinevar}
\sigma^2(\mathbf{f}) = \frac{1}{2rk}\mathbf{f}^t (\Delta_0^{-1})^tr^2
\mathbf{I} - A(\mathcal{ L}(G))^t A(\mathcal{ L}(G)) \Delta_0^{-1}
\mathbf{f}
\end{equation}

Lemma \ref{gradvanish} holds as well, and this gives us the following
useful corollary (a homological condition) about distribution on $G$
itself: 

\begin{theorem}
\label{cocycle}
The variance of a function $f$ on the vertices of $G$ vanishes,
precisely when there exists a function $g$, such that 
$\mathcal{L}f = \nabla g$.
\end{theorem}

Finally, we have a version of formula \ref{formula1}:

\begin{equation}
\label{dformula1}
\sigma^2(\mathbf{f}) = \frac{1}{2k} \left[\mathbf{f}^t \left(\mathbf{I} -
2r \Delta_0^{-1}\right) \mathbf{f}\right]
\end{equation}

\section{Distribution in compact groups}
\label{compgp}

The methods of the section \ref{primedist} can be adapted to the following
setting: Let $G$ is a graph, and $T$ be a compact topological
group. Label the $i$-th vertex of $G$ with $t_i \in T$. Now, associate
to each cycle $c=v_1, \dots, v_k$ on $G$ the element $t_c = t_k \cdot
\dots \cdot t_1 \in T$. We ask: as $c$ varies over the cycle space
$W_N$, how are the elements $t_c$ distributed in $T$ (with respect to
the Haar measure). The answer is given by the following:

\begin{theorem}
\label{cgp}
If the graph $G$ is as before (connected, non-bipartite), the closed
subgroup generated by the $t_i$ ($i=1, \dots, k$) is equal to $T$, and
the elements $t_i$ do not all lie in the same coset with respect to a
one-dimensional representation of $T$, then the elements $t_c$ become
equidistributed, as $N \rightarrow \infty.$
\end{theorem}

\begin{proof}
As before, the equidistribution is equivalent to the assertion that
for a {\em non-trivial} irreducible unitary representation $\rho$, 
\begin{equation}
\label{equiequiv}
\sum_{c\in W_n} \tr (\rho(t_c)) = o(|W_n|).
\end{equation}
This follows from the Fourier transform formula for compact groups;
see \cite{fulhar} for the finite case, \cite{weilinteg} for the general
compact topological group case. See also \cite{maslen}
Now, let $U(\rho)$ be the $k \deg \rho \times k \deg \rho$
block-diagonal matrix whose $j$-th block is just $\rho(t_j)$. 
Further more, as before, let $A(G)$ be the adjacency matrix of $G$,
and $A_l(G) = A(G) \otimes {\bf I}_l$ (where ${\bf I}_l$ is the
$l\times l$ diagonal matrix: in other words, $A_l(G)$ is a $kl \times
kl$ matrix, obtained from $A(G)$ by replacing each element $a_{ij}$ by
a $k\times k$ matrix $M_{ij}$, all of whose elements are equal to
$a_{ij}$. It is not hard to see that the left hand side of
Eq. \ref{equiequiv} is equal to $\tr (U(\rho) A_{\deg \rho}(G))^N,$
and so it suffices to show that the spectral radius of $M_\rho= U(\rho)
A_{\deg \rho}(G)$ is strictly smaller than the spectral radius of
$A(G)$ (which we normalize to be equal to $1$ by scaling) under the
hypotheses of the theorem. Suppose not. Since $(U(\rho))$ is unitary,
the worst that can happen is that there exists a unit vector $v$, such that
$\|M_\rho(v)\| = 1.$ If that is so, $v$ is contained in the eigenspace
of eigenvalue $1$ of $A_{\deg \rho}$. In such a case, $v = v_1 \otimes
u$, where $u \in V(\rho)$, and $v_1$ is an eigenvector of $A(G)$ with
eigenvalue $1$. If $v_1 = (v_1^1, \dots, v_1^n),$ then $v_1^i u$ must
be an eigenvector of $\rho(t_i)$, for all $i$. Since $v_1^i \neq 0$
$\forall i$, this implies that $u$ is an eigenvector $\rho(t_i)$,
$\forall i$. Since $\rho$ is irreducible, this implies that {\em
either} the elements $t_1, \dots, t_k$ do not generate all of $T$, or
$\rho$ is $1$-dimensional, in which case clearly
$\rho(t_i)=\rho(t_j)$, $\forall i, j$, which proves the theorem.
\end{proof}

\begin{remark}
\label{arbgp}
As in Remark \ref{arbcoeff}, the above argument also works if we pick
all paths between the $i$-th and the $j$-th vertex of $G$, instead of
all cycles.
\end{remark}

\section{Some perturbations and estimates}
\label{perturb}

Consider an analytic family of linear operators $M(x),$ acting on
$\reals^k,$ with $M(0)=M,$ and let $\lambda$ be a simple eigenvalue of
$M$. Then, if 
\[
M(x) = M + M^{(1)} x + M^{(2)} x^2 + \ldots,
\]
perturbation theory (see \cite[page 79, (2.33)]{kato95}) tells us that
\[
\lambda(x) = \lambda + \lambda^{(1)} x + \lambda^{(2)} x^2 + \ldots,
\]
where 
\begin{eqnarray}
\label{kato2}
\lambda^{(1)} &= \tr M^{(1)} P_\lambda,\\
\label{kato3}
\lambda^{(2)} &= \tr\left[M^{(2)} P_\lambda - M^{(1)} S_\lambda M^{(1)}
P_\lambda\right],
\end{eqnarray}
where  $P_\lambda$ is the projection onto the eigenspace of $\lambda,$
while $S_\lambda$ is the {\em reduced resolvent} of $M$ at $\lambda$,
which is the holomorphic part of the resolvent of $M$ at $\lambda$,
defined by the properties
\begin{equation}
\label{resolv1}
S_\lambda P_\lambda = P_\lambda S_\lambda = 0;\qquad (M-\lambda{\bf
I})S_\lambda = S_\lambda (M-\lambda{\bf I}) = {\bf I} - P_\lambda, 
\end{equation}
(in other words, $S_\lambda$ is the inverse of $M - \lambda{\bf I}$
restricted to the orthogonal complement of the eigenspace of $\lambda$),
and thus
\begin{equation}
\label{resolv2}
M S_\lambda = {\bf I} - P_\lambda + \lambda S_\lambda.
\end{equation}

Now we will specialize a bit:

\medskip\noindent
{\bf Assumption 1.} The eigenvalue $\lambda$ is such that
the constant vector ${\bf 1}$ spans the eigenspace of $\lambda.$ 

In this case, $P_\lambda$ = $J_k/k,$ where we recall that $J_k$ is the
$k\times k$ matrix of all $1$s.

In addition, 

\medskip\noindent
{\bf Assumption 2.} 
We will assume that $M(x)=D(x)M,$ where $D(x)$ is an
analytically varying diagonal matrix,
$D(x) = D + D^{(1)} x + D^{(2)} x^2 + \ldots,$ where we say that the
diagonal elements of $D^{(l)}$ are ${\bf d^{(l)}} = (d_1^{(l)},
\ldots, d_k^{(l)}).$

\begin{lemma}
\label{tracel}
Let $A=(A_{ij})$ be an $n\times n$ matrix. Then 
\[
\tr A J_n = \sum_{1\leq i, j\leq n} A_{ij}.
\]
\end{lemma}

\begin{lemma}
\label{conj}
Let $A=(A_{ij})$ be an $n\times n$ matrix, and let $X$ be an $n \times
n$ diagonal matrix. Then
\begin{eqnarray*}
(X A)_{ij} &= A_{ij} X_{ii},\\
(X A X)_{ij} &= A_{ij} X_{ii} X_{jj}.
\end{eqnarray*}
\end{lemma}

\begin{lemma}
\label{trace2}
Let $D$ be a diagonal matrix, with diagonal elements $d_1, \dots,
d_n$. Then  
\[v^t D v = \sum_{i=1}^n d_i v_i^2.\]
\end{lemma}

The proofs of the above lemmas are immediate.

\begin{lemma}
\label{trace3}
Let $P_v$ is the projection operator on the subspace generated by
$v$ (a unit vector). Then \[\tr M P_v = v^t M v.\]  
In particular, if $v$ is an eigenvector of $M$ with eigenvalue
$\lambda$, then $\tr M P_v = \lambda \| v \|.$
\end{lemma}

\begin{proof}
This follows by a direct computation, since when $v$ is a unit vector,
$(P_v)_{ij} = v_i v_j.$
\end{proof}

\begin{lemma}
\label{trace4}
If $v$ is an eigenvector of $M$ with eigenvalue $\lambda$, then $M P_v
= \lambda P_v$
\end{lemma}

\begin{lemma}
\label{var1}
Suppose that $\lambda$ has multiplicity $1$, and $v(\lambda)$ is a
unit vector generating the eigenspace of $\lambda$, and $M(t) = D(t)
M$, where $D(t)$ is a diagonal matrix. Then
\[\lambda^\prime(M) = \lambda v^t(\lambda) D^\prime v.\]
\end{lemma}

\begin{proof}
By Formula (\ref{kato2}), we have
\[\lambda^\prime(M) = \tr M^\prime P_{v(\lambda)} = v^t(\lambda)
M^\prime v(\lambda) = \lambda v^t(\lambda) D^\prime v.\]
\end{proof}
\begin{corollary}
In the case when $v(\lambda) = \\frac{1}{\sqrt{k}} mathbf{1}$, we have:
\begin{equation}
\label{firstvar}
\lambda^{(1)} = {\frac{\lambda}k} \sum_{j=1}^k d^{(1)}_j.
\end{equation}
\end{corollary}

To compute the second derivative of $\lambda$, we use the formula
(\ref{kato3}) (we are assuming that $\lambda$ is an isolated
eigenvalue with eigenvector $v(\lambda)$, and $M(t) = D(t) M$, as before):

\begin{eqnarray*}
\lambda^{\prime\prime} &= \tr \left[ M^{\prime\prime} P_{v(\lambda)} -
M^\prime S_\lambda M^\prime P_\lambda\right] \\
                &= \lambda v^t D^{\prime\prime} v - \tr\left[ M^\prime
S_\lambda M^\prime P_\lambda\right] \\
                &= \lambda v^t D^{\prime\prime} v - \lambda \tr
\left[D^\prime M S_\lambda D^\prime P_\lambda\right] \\
                &= \lambda v^t \left[D^{\prime \prime} - D^\prime M
S_\lambda D^\prime\right] v. 
\end{eqnarray*}
We can now use the formula (\ref{resolv2}) to get:

\begin{equation}
\label{secondvar}
\lambda^{\prime\prime} = \lambda v^t \left[D^{\prime \prime} -
D^\prime (\mathbf{I} - P_\lambda) D^\prime - \lambda D^\prime
S_\lambda D^\prime\right] v.
\end{equation}

In the special case where the eigenvector $v$ is proportional to
$\mathbf{1}$, we can rewrite the formula in coordinates in a simple
way. To wit, any diagonal matrix $D$ can be written (uniquely) as $D_0
+ d \mathbf{I}$, where $D_0$ is such that $\tr D_0 = 0.$ A simple
computation then shows that
\begin{equation}
\label{stochvar2}
\lambda^{\prime\prime} = \frac{\lambda}{k}\left[\sum_{j=1}^n
d_j^{\prime\prime} - \sum_{j=1}^n (d_0^\prime)^2 - \lambda
\mathbf{d}^{\prime t} S_{\lambda} \mathbf{d}^\prime\right].
\end{equation}

The case we are interested in is still more special, and that is where

\medskip\noindent
{\bf Assumption 3.}
\[
D(x) = \begin{pmatrix}
\exp(i f_1 x) &              &        & \cr
                              &\exp(i f_2 x) &        &  \cr
                              &              & \ddots &  \cr
                              &              &        & \exp(i f_k x)
\end{pmatrix}.
\]
Here, ${\bf d}^{(1)} = (i f_1, i f_2, \dots, i f_k),$
while ${\bf d}^{(2)} = -{\frac12}(f_1^2, f_2^2, \dots, f_k^2),$
and so, letting ${\bf f} = (f_1, \dots, f_k),$
\begin{equation}
\label{svharm}
\lambda^{(2)} = {\frac{\lambda}k} \left[ -{\frac12}\|{\bf f}\|^2 +
\|{\bf f}_0\|^2 + \lambda {\bf f}^t S_\lambda {\bf f} \right],
\end{equation}
where, as before, ${\bf f}_0$ is the component of ${\bf f}$ orthogonal
to constants.

To show our final estimates we shall need

\medskip\noindent
{\bf Assumption 4.}
The matrix $M$ is $\lambda > 0$ times a doubly stochastic matrix (this
implies that the operator norm and the spectral radius of $M$ are both
equal to $\lambda$).

\begin{theorem}
\label{posthmsym}
With assumptions as above, and, in addition, ${\bf
f} = {\bf f}_0$ (that is $\sum_{j=1}^k f_j = 0$), then $\lambda^{(2)}$ is
nonpositive.
\end{theorem}

\begin{proof}
Since ${\bf f}$ = ${\bf f}_0,$ Equation (\ref{svharm}) can be rewritten
as 
\begin{equation}
\label{svharm2}
\lambda^{(2)} = -{\frac{\lambda}{2k}} \left[-\|{\bf f}_0\|^2 - 2 \lambda
{\bf f}_0^\perp S_\lambda {\bf f}_0 \right] = -{\frac{\lambda}{2k}}
\left[{\bf f}_0^t \left(-{\bf I} - 2 \lambda S_\lambda \right) {\bf f}_0
\right]. 
\end{equation}
If we regard $S_\lambda$ as an operator on the orthogonal complement
to ${\bf 1},$ then by equations (\ref{resolv1}) and (\ref{resolv2}),
$S_\lambda (\lambda {\bf I}_0 - M_0) = -{\bf I}_0.$ Let $v=-S_\lambda
{\bf f}_0.$ Then the term in square brackets in Eq. \ref{svharm2} can
be rewritten as:
\begin{equation}
\label{svharm3}
v^t (\lambda {\bf I}_0 - M_0)^t \left(- {\bf I} - 2 \lambda S_\lambda
\right) (\lambda {\bf I}_0 - M_0) v = 
v_t \left(\lambda^2{\bf I}_0 - M_0^t M_0\right) v,
\end{equation}
where we have used the fact that for any matrix $A$ and any vector
$v$, $v^t A v = v^t A^t v.$ The quadratic form $\lambda^2{\bf I}_0 -
M_0^t M_0$ is positive semi-definite, since the biggest eigenvalue of
the symmetric matrix $M_0^t M_0$ is equal to the square of the operator
norm of $M_0$, which, in turn, is no greater then $\lambda$, by
Assumption 4 (since $M^tM$ is $\lambda^2$ times a doubly stochastic
matrix). 
\end{proof}

\begin{remark}
\label{sharpening}
In the statement of Theorem \ref{posthmsym}, the word ``non-positive''
can be improved to ``negative'' under the further assumption that $M$
is irreducible, primitive, and \textit{normal}.
\end{remark}

\begin{proof}
Since the orthogonal complement to the subspace generated by the
vector $\mathbf{1}$ is invariant under $M$, it follows that $M_0$ is
also normal, and so its operator norm is equal to its spectral
radius $\mu$. Under the assuptions of irreducibility and primitivity,
Perron-Frobenius theory tells us that $|\mu| < \lambda.$
\end{proof}

\section{Topological entropy}
\label{entropy}

Consider a graph $G$, and consider a positive function $f$ on its vertices. 
For each cycle $c$ we let $F(c)$ to be the sum of values of $f$ over
$c$, and we want to know how many $c$ are there for which $F(c) \leq
L$. We denote that number by $N(f, L)$, and we ask ourselves how $N(f,
L)$ behaves asymptotically as $L$ tends to infinity. To understand
$N(L, f)$, we consider first the matrix $U(f) = D(u^{f_1}, \dots,
u^{f_n}) A(G).$ As before, we observe that the coefficient of $u^r$ in
$\tr U^n(f)$ is the number of cycles of (combinatorial) length $n$,
for which $F(c) = r$. Write a formal series
\begin{equation*}
L(f, u) = \sum_n \tr U^n(f). 
\end{equation*}
This series converges for sufficiently small $u$, and can there be
written in closed form as $L(f, u) = \tr (\mathbf{I} - U(f))^{-1}$,
from which it follows that the exponential rate of growth of $N(c)$ is
equal to negative logarithm of the radius of convergence of $L(f, u)$
-- we call this the \textit{entropy of $G, f$} -- 
which, in turn, is equal to the smallest positive real value of $u$,
such that the spectral radius of $U(f)$ is equal to $1$. Since it is
more convenient to deal with analytic functions (which $L(f, u)$ is
not, for arbitrary real values of $f_i$, so we write $u = \exp -s,$
and now ask for the abscissa of convergence of $L(f, \exp -s)$. This
will give us the entropy. In this section we use perturbation methods
in a rather straightforward way to get explicit information on the entropy.

Let $A$ be an $n\times n$ non-negative primitive irreducible
matrix. Let $f_1, \dots, f_n$ be a collection of weights. We then
define the matrix $E(s, \mathbf{f})$ to be the diagonal matrix whose
$ii$-th element is equal to $\exp(-s f_i)$. Define $M(s, \mathbf{f})$
to be $M(s, \mathbf{f}) = E(s, \mathbf{f}) A.$ We are interested in
$\rho(s, \mathbf{f})$: the spectral radius of $M(s, \mathbf{f})$.
By Perron-Frobenius theory we know that there is a real eigenvalue of
$M(s, \mathbf{s})$ equal to $\rho(s, \mathbf{f})$, and  the
eigenvector $v_\rho$ of this eigenvalue is positive.

\begin{lemma}
\label{monotonicity}
\begin{equation}
\label{monoeq}
\frac{\partial \rho}{\partial s} = - \rho v_\rho^t D(f_1, \dots, f_n)
v.
\end{equation}
For positive $\mathbf{f}$, $\frac{\partial \rho}{\partial s} < 0.$
\end{lemma}

\begin{proof}
This follows immediately from Lemma \ref{trace3} and the positivity of
$\rho$ and $v_\rho$. 
\end{proof}

\begin{lemma}
\label{gradient}
We have the following expression for the gradient of $\rho$ with
respect to $\mathbf{f}$:
\begin{equation}
\label{gradeq}
\nabla_{\mathbf{f}} \rho = -s \rho (v_1^2, \dots, v_n^2),
\end{equation}
where $v_{\rho} = (v_1, \dots, v_n).$
\end{lemma}

\begin{proof}
We note that 
\begin{equation*}
\frac{\partial M}{\partial f_i} = - s D(0, \dots, 1, \dots, 0) M, 
\end{equation*}
where the $1$ is in the $i$-th place.
Thus, by formula (\ref{kato2}) we have
\begin{equation*}
\frac{\partial \rho}{\partial f_i} = -s v_\rho^t D(0, \dots, 1, \dots,
0) M v = -s \rho v_i^2.
\end{equation*}
\end{proof}

This can be restated as saying that the derivative of $\rho$ in the
direction of a vector $\mathbf{g}$ is equal to $- \rho s v_\rho^t
D(\mathbf{g}) v.$

This gives us the following important corollary:

\begin{corollary}
\label{constv}
Consider deformations $g$ keeping the sum of $f_i$ fixed. Then the
critical points of $\rho$ occur precisely for those $\rho$ for which
$|v_i| = |v_j|$, for any $i, j$.
\end{corollary}

We can also compute the second directional derivative of $\rho$.
Indeed, let $\mathbf{g} = (g_1, \dots, g_n)$ be the direction vector,
so that we want to compute the second derivative with respect to $t$
of $\rho(s, \mathbf{f} + t \mathbf{g})$ at $t = 0.$
To do this, we use the formula (\ref{kato3}):
\begin{equation}
\label{kkato}
\rho^{\prime\prime} = \tr \left[M^{\prime\prime} P_{v(\rho)} -
M^\prime S_\rho M^\prime P_\rho\right].
\end{equation}
Note that (as in the proof of Lemma \ref{gradient})
\begin{equation}
\label{grad1}
M^\prime = - s D(g_1, \dots, g_n) M,
\end{equation}
while
\begin{equation*}
M^\prime\prime = s^2 D(g_1^2, \dots, g_n^2) M,
\end{equation*}
and so
\begin{equation}
\label{term1}
\tr M^{\prime\prime} P_{v(\rho)} = s^2 \rho v^t D(g_1^2,\dots, g_n^2)
v = s^2 \rho \left\{D(g_1, \dots, g_n) v\right\}^t \left\{D(g_1,
\dots, g_n) v \right\}.
\end{equation}
To understand the second term of the right-hand side of
Eq. (\ref{kkato}), first note that (by Eq. (\ref{grad1}))
\begin{equation*}
M^\prime S_\rho M^\prime P_\rho = 
D(g_1, \dots, g_n) M P_\rho = \rho s^2D(g_1, \dots, g_n)  M S_\rho
D(g_1, \dots, g_n) P_\rho,
\end{equation*}
where the second equality is by Lemma \ref{trace4}.
Now 
\begin{eqnarray}
\label{term2}
\tr M^\prime S_\rho M^\prime P_\rho &= \rho s^2 v(\rho)^t D(g_1, \dots,
g_n) M S_\rho D(g_1, \dots, g_n) v\\
 &= \rho s^2 \left\{D(g_1, \dots,
g_n) v\right\}^t M S_\rho \left\{ D(g_1, \dots, g_n) v\right\}.
\end{eqnarray}
Putting together Eq. (\ref{term1}) and Eq. (\ref{term2}), we see that 
\begin{equation}
\label{twoterms}
\rho^{\prime\prime} = \rho s^2 \left\{ D(g_1, \dots, g_n) v\right\}^t
\left(\mathbf I - M S_\rho\right) \left\{ D(g_1, \dots, g_n) v\right\}
\end{equation}

Using the formula (\ref{resolv2}) equation (\ref{twoterms}) simplifies
further to:
\begin{equation}
\label{twomoreterms}
\rho^{\prime\prime} = \rho s^2 \left\{ D(g_1, \dots, g_n) v\right\}^t
\left(P_{v(\rho)} - \rho S_\rho\right) \left\{ D(g_1, \dots, g_n) v\right\}
\end{equation}

The following lemma is not surprising:
\begin{lemma}
\label{posdefres}
The quadratic form given $P_v - \rho S_\rho$ is positive-definite.
\end{lemma}

\begin{proof}
On the span of $v$, the projection operator $P_v$ is equal to the
identity, whilst the reduced resolvent $S_\rho$ vanishes. On the
orthogonal complement, the projection operator vanishes, so since the
Perron-Frobenius eigenvalue $\rho$ is positive, we need to show that
$S_\rho$ is negative definite. Consider a vector $w$, in the
orthogonal complement of $v$. Such a $w$ is equal to $(\rho \mathbf{I}
- M) z,$ for some $z$ orthogonal to $v$. So, 
\begin{equation*}
w^t S_\rho w = z^t (\rho I - M) z,
\end{equation*}
So, it will suffice to show that $(\rho I - M)$ is negative-definite.
Suppose not. Then there exists a $z_0$, such that $z_0^t M z_0 \geq
\rho \|
z_0\|^2.$ By the argument in the proof of theorem \ref{posthmsym}, we
see that $\|M z_0\| \leq \rho \|z_0\|$. So,  $z_0^t M z_0 \geq
\rho \|z_0\|^2$  implies that $\langle z_0, M z_0\rangle \geq
\rho \|z_0\|^2,$ and hence that $z_0$ is an eigenvactor of $M$ with
eigenvalue $\rho$, which is impossible by assumtion that $M$ is
irreducible and primitive.
\end{proof}

We finish with
\begin{theorem}
\label{minent}
Let $s_0(\mathbf{f})$ be the unique $s$ such that $\rho(s_0,
\mathbf{f})$ is equal to $1$. Then $s_0$ is a convex function of
$\mathbf{f}$, and hence assumes a unique minimum on each linear
subspace of values of $\mathbf{f}$. In particular, if we restrict to
the the subspace $F_0$, where the sum of the values of of $\mathbf{f}$
is equal to $1$, then the minimum is achieved at the point 
where 
\begin{equation}
f_i = \frac{\log (A \mathbf{1})_i}{\sum_i \log (A \mathbf{1})_i},
\end{equation}
in which case the entropy is equal to $\sum \log (A \mathbf{1})_i$.
\end{theorem}

\begin{proof}
The convexity of $s_0$ follows from Lemma \ref{posdefres} and Lemma
\ref{monotonicity}. The point at which the minimum is achieved is
computed easily using Corollary \ref{constv}, as is the value of
entropy.
\end{proof}

\section{Applications to Groups and other objects}
\label{appsec}

The asymptotic results in the previous sections apply directly to the
question of the growth of homology classes in the free groups, and
give in some sense complete information:

\begin{observation}
We see that the asymptotic order of growth of any two {\em
fixed} homology classes is the same.
\end{observation}

\begin{observation}
Theorem \ref{walks} shows in particular that a
random long cycle is equidistributed among the vertices of a regular
graph. 
\end{observation}

\begin{observation}
We see that the order of growth the number of words length
$n$ in any fixed homology class
in $F_k$ is asymptotic to $c_k (2k-1)^n/n^{k/2},$ where $c_k$ is
easily computed using the expression for $\sigma$ in the statement of
Theorem \ref{centlim}, keeping in mind that 
\[
c_{F_k} = {\frac{k}{\sqrt{2k - 1}}},
\]
where $c$ is the parameter in the statements of theorems of the last
two sections. Alternately, Theorem \ref{walks} can be used.

We can compute other growth functions. For example, let $h:
F_n \rightarrow \integers$ be the ``total exponent'' homomorphism, {\em
i.e.} if $F_n = <a_1, \dots, a_n>,$ then $h(a_i) = 1.$ We see that the
generating function for the preimages of $j\in Z$ is given by 
\[
\left(2 \sqrt{2n-1}\right)^k 
R_k({\frac{n}{\sqrt{2n-1}}}; x, \dots, x) = 
\left(2 \sqrt{2n-1}\right)^k R_k({\frac{n}{\sqrt{2n-1}}}; x).
\]
\end{observation}

\begin{observation}
Instead of cyclically reduced words, it is perhaps more
natural to study conjugacy classes (ordered by their cyclically
reduced length). It seems futile to seek any enumeration as neat as
Theorem \ref{homoenum}, however, since the relationship between the
number $\mathcal{ C}_k$ of conjugacy classes of words of length $k$ and
the number of cyclically reduced words $\mathcal{ W}_k$ is: 
\begin{equation}
\label{cclasses}
\mathcal{C}_k = {\frac{\calw_k}{k}} + O(\sqrt{\calw_k}),
\end{equation}
it is clear that the asymptotic results are the same for the two
problems. For more on this subject, see Section \ref{gf} and the
sequel. 
\end{observation}

\begin{observation}
Counting conjugacy classes is a problem closely related to
that of counting closed geodesics on manifold. In the context of
compact hyperbolic surfaces, it was observed by P.~Sarnak (see, for
example, \cite{sarnakrubinstein})  that among all geodesics shorter than $L,$
null-homologous geodesics are more numerous than those in any other
prescribed homology class (that is, while the ratio of the two
quantities approaches $1$, the difference is asymptotically positive). 
The results of the current note provide a certain justification for
this, since any limiting distribution likely to arise in this context
is, for reasons of symmetry, likely to be unimodal, with the mode at
${\bf 0}.$ Certainly this is true of the normal distribution, though
even in this case, a careful analysis of the error terms is required.
\end{observation}

\section{Counting conjugacy classes}

\label{gf}

Consider a finitely presented group $G.$ Let $g$ be an element of
$G$. We define the \emph{reduced length} of $g$ -- denoted by $|g|$ --
to be the length of the shortest word in the generators of $G$
representing $g$. We define the \emph{length up to conjugacy} of
$g$ -- denoted by $|g|_c$ -- to be the minimum of $|h|$, the minimum
being taken over all group elements $h$ conjugate to $g$. 
Length up to conjugacy is obviously invariant under conjugation,
and we will also use the term to apply to conjugacy classes.

\begin{gather*}
\mathcal{ N}_G(r) = \left|\{g\in G \bigm| \quad |g| = r\} \right|,\\
\mathcal{ C}_G(r) = \left|\{g\in \mathcal{ N}_G(r) \bigm| |g|_c = r\} \right|,\\
\mathcal{ C}\mathcal{ C}_G(r) = \left|\{C \in G/\mbox{conjugacy} \bigm| |C|_c =r\} \right|.
\end{gather*}

The subscript $G$ will be omitted whenever the group $G$ is obvious
from context.

Given a sequence $A=a_0, \ldots, a_i, \ldots$, we can define a
\emph{generating function} $\mathcal{ F}[A]$, by 
\[\mathcal{ F}[A](z)=\sum_{i=0}^\infty a_i z^i.\]
There is frequently confusion as to whether the generating function is
a holomorphic function or an element of the ring of formal power
series. In this section ``generating function'' will mean a function
analytic at $0 \in {\bf C}$. 

The three counting functions above give rise to corresponding
generating functions $\mathcal{ F}[\mathcal{ N}_G]$, $\mathcal{ F}[\mathcal{ C}_G]$,  
$\mathcal{ F}[\mathcal{ C}\mathcal{ C}_G]$. Our real interest will lie in the last
of these; the first one has been the most extensively studied, and the
result most relevant to us is:

\medskip\noindent
{\bf Fact 1.} If $G$ is an \emph{automatic} group, then the generating
function $\mathcal{ F}[\mathcal{ N}_G]$ is a rational function.

For definitions and properties of automatic groups, see
\cite{wordprocessing}.

\medskip\noindent
{\bf Fact 2.}(Gromov, Epstein) If $G$ is an automatic group, then the
generating function $\mathcal{ F}[\mathcal{ C}_G]$ is a rational function.

Facts 1 and 2 might lead us to expect that $\mathcal{ F}[\mathcal{ C}\mathcal{
C}_G]$ is, likewise, rational, but in fact the opposite seems to be
the case, and we are led to:

\begin{conjecture}
\label{conj1}
Let $G$ be a word-hyperbolic group. The $\mathcal{ F}[\mathcal{ C}\mathcal{ C}_G]$
is rational if and only if $G$  is virtually cyclic ({\em
elementary} in the terminology of \cite{gromovgroups}).
\end{conjecture}

In the sequel, this conjecture is supported by the complete
analysis of the case where $G$ is $F_k$ -- the free group on $k$
generators.

\section{Growth functions for free groups}
\label{gf2}

Let $F_k$ be the free group on $k$ generators. The following is
obvious:

\medskip\noindent{Fact 3.} $\mathcal{ N}_{F_k}(r) = 2 k (2k-1)^{r-1}.$

Theorem \ref{cycredno} says that \[
\mathcal{ C}_{F_k}(r) = (2k-1)^r + 1 +
(k-1)[1+(-1)^r].\]

\begin{corollary}
\label{expcor}
\[
\mathcal{ F}[\mathcal{ C}_{F_k}](z) = {\frac1{1-(2k-1)z}} + {\frac1{1-z}} +
{\frac{2(k-1)}{1-z^2}} - 2k. 
\]
\end{corollary}
In order to compute $\mathcal{ C}\mathcal{ C}_{F_k}(r)$ it is enough to notice
the following:

\begin{theorem}
\label{toti}
\[ r \mathcal{ C}\mathcal{ C}(r) = \sum_{d\bigm|r} \phi(d) \mathcal{ C}(r/d),\]
where $\phi$ denotes the Euler totient function.
\end{theorem}

\begin{proof}
The theorem is a trivial consequence of Burnside's lemma, stated below
as Theorem \ref{burn} for convenience, applied to the action of the
cyclic group ${\bf Z}/(r {\bf Z})$ on the set of cyclically reduced
words of length $r$.
\end{proof}

\begin{theorem}
\label{burn}
Let $G$ be a finite group acting on a finite set $X$. For $g\in G$ let
$\psi(g)$ denote the number of $x\in X$, such that $g(x) = x.$ Then
the number of orbits of $X$ under the $G$-action is
\[
\frac{1}{|G|} \sum_{g\in G} \psi(g).
\]
\end{theorem}

We now have the following general observation:

\begin{theorem}
\label{genthm}
Suppose we have three sequences $A=\{a_i\}$, $B = \{b_j\},$
and $C= \{c_k\},$ satisfying

\[a_n = \sum_{d\bigm|n} c_d b_{\frac{n}{d}}.\]
Then 

\[\mathcal{ F}[A](z) = \sum_{d=1}^\infty c_d \mathcal{ F}[B](x^d).\]
\end{theorem}

\begin{proof}
On the level of formal power series, the statement is
clear by expanding the left hand side. Otherwise, if the radius of
convergence of $\mathcal{ F}[A]$ is $r_a$, 
then the radius of convergence of $G_d[A]$, defined as $G_d[A](z) =
\mathcal{ F}[A](z^d)$ is, by Hadamard's criterion, equal to $r_a^{1/d}$, so all
of $G_d[A]$ converge on the disk of radius $R_a=\min(r_a, 1)$ around the
origin. Since the series on the right hand side converges at $0$
(since all the terms vanish), it converges uniformly on compact
subsets of the disk of radius $R_a$ around the origin.
\end{proof}

\begin{corollary}
\label{stupcor}
Let $\mathcal{ H}$ be the generating function of the sequence $h_r = r
\mathcal{ C}\mathcal{ C}(r).$ Then
\[
\mathcal{ H}(z) = 1+ \sum_{d=1}^\infty \phi(d)\mathcal{ F}[\mathcal{ C}](z^d).
\]
\end{corollary}

We can combine all of the above results into the following conclusion:
\begin{theorem}
\label{finthm}
The generating function $\mathcal{ H}$ as in the statement of corollary
\ref{stupcor} can be expanded as:

\[
\mathcal{ H}=1 +  (k-1)\frac{x^2}{(1-x^2)^2} + 
\sum_{d=1}^\infty \phi(d)\left(\frac{1}{1-(2k-1)x^d}-1\right).
\]

In particular, $\mathcal{ H}$ has an infinite number of poles, and is not
a rational function for any $k>1.$ The generating function $\mathcal{ F}[\mathcal{ C}\mathcal{
C}_{F_k}]$ can be written as 
\[\mathcal{ F}[\mathcal{ C}\mathcal{ C}_{F_k}](z) = \int_0^z  \frac{\mathcal{
H}(t)}{t} d t\]
and so is not a rational function either.
\end{theorem}

\begin{proof}
The expression for $\mathcal{ H}$ is fairly obvious, with the comment that
the second summand is a consequence of the fact that 
\[\sum_{d\bigm|n} \phi(d) = n.\]That $\mathcal{ H}$ has an infinite number of
poles follows from the observation that the $d$-th term in the third
summand has its $d$ poles on the circle $|z| = (2k-1)^{-1/d}$, while the
first two summands are analytic in the open unit disk.
The expression for $\mathcal{ F}[\mathcal{ C}\mathcal{ C}_{F_k}]$ is immediate.
\end{proof}

\begin{remark}
For $k=1,$ it is not hard to see that
\[
\mathcal{H}=1+\dfrac{x}{(x-1)^2}.
\]
\end{remark}

\medskip\noindent
{\bf Remark.} Various people, when shown Theorem \ref{finthm},
appeared to believe that it contradicts \cite[Theorem
5.2D]{gromovgroups}. In fact (as pointed out by Greg McShane), Gromov's
function $[N]_k$ is {\em not} (as the common misunderstanding has it)
the same as $\mathcal{ C}\mathcal{ C}_G(r)$ in the case of a free group, but
{\em is} the same as $\mathcal{ C}_G(r)$. 

\subsection{Some further comments}
\label{gf3}

The following observation is quite obvious:

\begin{observation}
\label{prodthm}
Let $G_1$ and $G_2$ be two groups. Then, 
\[
\mathcal{ F}[\mathcal{ C}\mathcal{ C}_{G_1 \times G_2}](z) = \mathcal{ F}[\mathcal{
C}\mathcal{ C}_{G_1}](z) \mathcal{ F}[\mathcal{ C}\mathcal{ C}_{G_2}](z).
\]
\end{observation}
It would be interesting to find other relationships (for example, what happens for HNN extensions?)

Observation \ref{prodthm} has some consequences:

\thm{theorem}
{
\label{corfin}
Let $G_1$ and $G_2$ be two groups, then if 
$\mathcal{ F}[\mathcal{ C}\mathcal{ C}_{G_1}]$ is rational, while 
$\mathcal{ F}[\mathcal{ C}\mathcal{ C}_{G_2}]$ is not, then 
$\mathcal{ F}[\mathcal{ C}\mathcal{ C}_{G_1 \times G_2}]$ is not rational.
If both $\mathcal{ F}[\mathcal{ C}\mathcal{ C}_{G_1}]$ and
$\mathcal{ F}[\mathcal{ C}\mathcal{ C}_{G_2}]$ are rational, then so is
$\mathcal{ F}[\mathcal{ C}\mathcal{ C}_{G_1 \times G_2}]$.
}

\begin{corollary}
\label{cor2}
If $G_1 = \mathbf{Z}^n$ and $G_2$ is a finite group, then $\mathcal{
F}[\mathcal{ C}\mathcal{ C}_{G_1 \times G_2}]$ is rational.
\end{corollary}

\medskip\noindent
{\bf Remark.} It is not clear whether $\mathcal{ F}[\mathcal{ C}\mathcal{ C}_{G}]$
is rational when $G$ is a Bieberbach group -- most likely this depends
on the choice of the generating set, as conjectured by
D.~B.~A.~Epstein.

\begin{corollary}
\label{cor3}
If $G_1 = F_k$ and $G_2$ is a direct product of finite groups and
infinite cyclic groups, then $\mathcal{
F}[\mathcal{ C}\mathcal{ C}_{G_1 \times G_2}]$ is irrational.
\end{corollary}

\begin{theorem}
\label{thm4}
If $G = F_{k_1} \times F_{k_2} \times \ldots \times F_{k_n}$, then 
$\mathcal{ F}[\mathcal{ C}\mathcal{ C}_G]$ is irrational (with respect to the
``obvious'' generating set).
\end{theorem}

\begin{proof} This is an immediate consequence of Theorem \ref{finthm}.
\end{proof}

\section{Primitive conjugacy class zeta function}
\label{ihara}
One can compute a zeta-function analogous to that of
Ihara for the numbers of {\em primitive} conjugacy classes of a given
length (a primitive class is one which is not the power of a smaller
class), using, essentially, the elementary method described by Stark
and Terras, \cite{staterI}, as applied to the graph constructed in Section
\ref{modsec}. This function turns out to be rational (in fact, there
is a simple formula for it, see Theorem \ref{zetadet}). More precisely,
consider 
\begin{equation}
\label{zetadef}
\zeta(G)^{-1}=\prod_{[c]} (1+u^l(c)),
\end{equation}
where $[c]$ denotes the equivalences classes of primitive cycles,
where two cycles are considered equivalent if one can be obtained from
the other by a rotation.

A computation then shows that 
\begin{equation}
\label{zetaform}
\zeta(F_r) = (1-u^2)^{r-1}(1-u)(1-(2r-1)u).
\end{equation}

The computation goes as follows:

First, note that 
\[
\log \zeta(G) = \sum_{[c]} \sum_{i=1}^\infty \frac{1}{i} u^{i l(c)},
\]
and thus
\[
u \frac{d \log \zeta(G)}{d u} = \sum_{[c]} \sum_{i=1}^\infty l(c) u^{i
l(c)}.
\]
The above can be rewritten (note that the sum is now over primtive
cycles, and not equivalence classes thereof):
\[
u \frac{d \log \zeta(G)}{d u} = \sum_{c} \sum_{i=1}^\infty u^{i l(c)}.
\]
But note that the right hand side is simply the ordinary  generating function
for {\em all} cycles:
\begin{equation}
\label{cycleeq}
u \frac{d \log \zeta(G)}{d u} = \sum_{i=1}^\infty N_i u^i,
\end{equation}
where $N_i$ is the number of cycles of length $i$ in $G$, and this
generating function was computed in Section \ref{modsec}:
\[
\sum_{i=i}^\infty N_i u^i = \frac{1}{1+(2r-1)u} + \frac{r}{1-u} +
\frac{r-1}{1+u}.
\]
The formula \ref{zetaform} now follows by a straightforward
integration. 

An quick examination of the above argument shows that the formula
\ref{zetaform} is a special case of the following result:

\begin{theorem}
\label{zetadet}
Let $G$ be a finite graph, and let $\zeta_G$ be the zeta function
defined by formula \ref{zetadef}. Let $A(G)$ be the adjacency matrix
of $G$. Then 
\begin{equation}
\zeta_G(u) = \det \left(I-u A(G)\right).
\end{equation}
In other words, the zeta function is essentially the characteristic
polynomial of $A(G)$.
\end{theorem}

\begin{proof}
The argument above up to Equation \eqref{cycleeq} is completely general.
On the other hand, the right hand side of Equation \eqref{cycleeq} can
be rewritten as:
\begin{align*}
\sum_{i=1}^\infty N_i u^i & = \sum_{i=1}^\infty \tr A(G)^i u^i \\
                          & = \tr \left[-I + \sum_{i=0}^{\infty}
                          \left[A(G)u\right]^i\right] \\
                          & = \tr \left[-I + (I - u A(G))^{-1}\right] \\
                          & = \tr \left(u A(G)(I-u A(G))^{-1}\right). 
\end{align*}
Thus,
\[
\frac{d \log \zeta(G)}{d u} = \tr \left(A(G)(I-uA(G))^{-1}\right),
\]
and so it follows that
\begin{equation*}
\zeta(G) = C \det(I-u A(G)), 
\end{equation*}
where $C$ is a constant of integration, seen to be equal to $1$ by
computing both sides at $u=0.$
\end{proof}

\subsection*{Acknowledgements}
I would like to thank those who had comments on 
earlier versions of this paper (released as an IHES preprint in
September 1997). The irrationality of the growth function of the number
of conjugacy classes was independently shown by D.~B.~A.~Epstein and
Murray Macbeath. I would like to thank the anonymous referees for their comments on the current version.

\bibliographystyle{plain}
\bibliography{maslen,rivin1}
\end{document}